\documentclass[smallextended,numbook,envcountsame,envcountreset,draft]{svjour3}

\usepackage{url}

\usepackage{amsmath,amssymb}
\usepackage{color}

\newcommand{\dd}{\mathrm{d}}

\DeclareMathOperator{\dom}{dom}

\DeclareMathOperator{\sgn}{sgn}

\renewcommand{\theenumi}{(\roman{enumi})}

\newtheorem{assumption}{Assumption}[section]

\smartqed

\begin{document}
\title{A Semigroup Point Of View On Splitting Schemes For Stochastic (Partial) Differential Equations
\thanks{The first author gratefully acknowledges partial support by the FWF grant W8. Financial support from the ETH Foundation is gratefully acknowledged.}
}
\titlerunning{A Semigroup Point Of View On Splitting Schemes For S(P)DEs}
\author{
Philipp D\"orsek
\and 
Josef Teichmann}
\institute{Philipp D\"orsek \at Vienna University of Technology, Wiedner Hauptstra\ss e 8-10, A-1040 Vienna, Austria,
\\\email{philipp.doersek@tuwien.ac.at}
\and
Josef Teichmann \at ETH Z\"urich, D-MATH, R\"amistra\ss e 101, 8092 Z\"urich, Switzerland,
\\\email{jteichma@math.ethz.ch}
}

\maketitle

\begin{abstract}
	We construct normed spaces of real-valued functions with controlled growth on possibly infinite-dimensio\-nal state spaces such that semigroups of positive, bounded operators $(P_t)_{t\ge 0}$ thereon with $\lim_{t\to 0+}P_t f(x)=f(x)$ are in fact strongly continuous. This result applies to prove optimal rates of convergence of splitting schemes for stochastic (partial) differential equations with linearly growing characteristics and for sets of functions with controlled growth. Applications are general Da Prato-Zabczyk type equations and the HJM equations from interest rate theory.
\end{abstract}

\section{Introduction}

In applications, we often apply mathematical theory to models,
even though the assumptions of the respective theory are not completely satisfied. For instance, when we consider
the Heston stochastic volatility model, it is clear that the involved vector fields are not everywhere Lipschitz continuous on the state space, and that linearly, not to mention exponentially growing payoffs, do not fall into the class of test functions where a guaranteed rate of convergence is provided. Nevertheless we do not hesitate to apply Euler or higher order schemes and we clarify -- if we have time -- the raised open questions in a case-by-case study.

There arises the interesting and promising question of whether there is a general statement possible that embeds different specialised results into a general framework.
In particular in infinite dimension this larger picture is fairly unknown.
This is due to the additional phenomenon of unboundedness, which is impossible to circumvent in concrete cases, see for instance \cite{BayerTeichmann2008}.

In this work, we want to provide this larger picture for splitting schemes for S(P)DEs.
This allows us to deal with unbounded payoff functions and with certain kinds of singularities of the local characteristics. 
It is well-known that the world of stochastic Markov processes on general state spaces is tied to strongly continuous semigroups in two ways: either through the Feller property, or through invariant measures. 
In both cases we can construct an appropriate Banach space, $C_0(X) $ and $ L^p(X,\mu) $, respectively, where the Markov semigroups act in a strongly continuous way. 
Strong continuity is in many senses a ``via regia'' towards approximation schemes via splitting schemes (e.g.~Trotter-type formulae, Chernov's theorem, etc), and therefore a very desirable feature. 
However, neither the existence of invariant measures nor the Feller property are generic properties of Markov processes -- this holds true in particular in infinite dimension. 
The situation is even worse for the Feller property, where we have a strong connection to locally compact state spaces and continuous functions vanishing at infinity.
It therefore seems natural to ask for a framework extending the Feller property towards unbounded payoffs and non-locally compact spaces. 
Moreover, the framework should be as generic as possible to remain applicable to general SPDEs. 
From the viewpoint of applications, the new concept is useful if we are able to prove rates of convergence for substantially larger classes of payoffs and equations with the presented method.

Let us first fix what we mean by a \emph{splitting scheme} for Markov processes (cf.~\cite{TanakaKohatsuHiga2009} for a similar, abstract approach, or \cite{TalayTubaro1991} for a more concrete approach, both in the finite dimensional setting). 
Let $ x(t,x_0) $ be a Markov process on a (measurable) state space $ X $ and assume that
\begin{itemize}
	\item there is a (some) Banach space $ \mathcal{B}(X) $ of real-valued functions with Markov semigroup $ P_t f(x_0) := \mathbb{E}[f(x(t,x_0))] $, for $ f \in \mathcal{B}(X)$,  $ t \geq 0 $ and $ x_0 \in X $, acting on $ \mathcal{B}(X) $ as a semigroup of linear operators bounded by $ M \exp(\omega t) $ for some $ M \geq 1 $ and some real $ \omega $;
	\item there are semigroups $ P^{(1)}, \ldots, P^{(k)} $ of linear operators on $ \mathcal{B}(X) $ such that the weighted composition
		\begin{equation}
			Q_{(\Delta t)}  := \sum_{j=1}^K \lambda_j P^{(i_1)}_{\delta^j_1 \Delta t} \dots  P^{(i_l)}_{\delta^j_l \Delta t}, 
		\end{equation}
		for some real numbers $ \delta^j_i \geq 0 $ and $ \Delta t > 0$, and some weights $ \lambda_j \geq 0 $, form a family of operators power-bounded on some interval $ [0,T] $, in the sense that $ {(Q_{(\Delta t)})}^{m} $ is bounded in operator norm for all $ 1 \leq m \leq n $ and $ n\Delta t \in [0,T] $; and
	\item the short time asymptotic expansions of order $ p>1 $ of the operators $P_{\Delta t}$ and $Q_{(\Delta t)}$ coincide on some subspace $ \mathcal{M} \subset \mathcal{B}(X) $, i.e.~ 
		\begin{equation}
			\lVert P_{\Delta t} P_s f - Q_{(\Delta t)} P_s f \rVert \leq C_{f} \Delta t^p 
		\end{equation}
		for $ f \in \mathcal{M} $ and $ s, \Delta t \in [0,T]$.
\end{itemize}
Under these assumptions we can readily prove that
\begin{equation}
	\lim_{n \to \infty} {(Q_{(\frac{t}{n})})}^n f = P_t f
\end{equation}
for $ f \in \mathcal{M} $ and $t\le T$. The proof is well-known and simple due to the telescoping sum
\begin{equation}
	{(Q_{(\frac{t}{n})})}^n f - P_t f  = \sum_{i=1}^{n-1} {(Q_{(\frac{t}{n})})}^{(n-i)} \bigl ( Q_{(\frac{t}{n})} -  P_{\frac{t}{n}}) P_{\frac{it}{n}} f
\end{equation}
for $ t \in [0,\varepsilon] $ and $ f \in \mathcal{M} $. 
We even obtain weak convergence of order $p-1$ on $\mathcal{M}$, i.e.
\begin{equation}
	\lVert P_t f - (Q_{(\frac{t}{n})})^n f\rVert
	\le
	C_f \bigl(t/n)^{p-1}.
\end{equation}
Due to the boundedness properties of the involved operators the convergence extends to the closure of $ \mathcal{M} $.
The rate of convergence, however, is then lost.
While splitting schemes as formulated above have an order bound for positive step sizes \cite{BlanesCasas2005} and the choice of $\delta$ and $\lambda$ in the Ninomiya-Victoir splitting \cite{NinomiyaVictoir2008} thus yields the optimal possible order, the above approach can also be taken for approximations $Q_{(\Delta t)}$ which are not necessarily derived from a splitting scheme. 
The authors use similar methods to derive rates of convergence for cubature methods for stochastic partial differential equations in a forthcoming paper.

Using Lyapunov-type functions $ \psi $, we shall construct Banach spaces $ \mathcal{B}^{\psi}(X) $ where the previous requirements are satisfied for Euler- and Ninomiya-Victoir-type schemes. 

Even in finite dimensions this is -- in its generality -- a new result and
can be seen as widening \cite{Kusuoka2001} to the case of unbounded
coefficients and unbounded claims, further extending the work from
\cite{TanakaKohatsuHiga2009}. Its importance, however, lies in its
applicability to problems with infinite dimensional state spaces. We achieve
this in a unified way by putting the theory of \cite{TanakaKohatsuHiga2009}
on an abstract theoretical basis through developping a notion of generalised
Feller semigroups.

We outline our ideas in a finite dimensional example, but it is the goal of this work to show that a corresponding result can also be achieved for SPDEs.

\begin{example}
	Consider a stochastic differential equations
	\begin{equation}
		\dd x(t,x_0) = V (x(t,x_0)) \dd t + V_1(x(t,x_0)) \dd B_t
	\end{equation}
	with $ \mathrm{C}^3 $-bounded vector fields $ V, V_1 $ driven by a one-dimensional Brownian motion $(B_t)_{t\ge 0}$.
	It is well known that we can consider the Markov process $ x(t,x_0) $ and its semigroup $ (P_t)_{t\ge 0} $ on the space $ \mathrm{C}_0(\mathbb{R}^N) $ of continuous functions decaying at infinity, endowed with the norm $\lVert f\rVert_{\infty}:=\sup_{x\in\mathbb{R}^N}\lvert f(x)\rvert$. 
	Since $ \lvert P_t f(x_0)\rvert = \lvert \mathbb{E}[f(x(t,x_0))]\rvert \leq \lVert f \rVert_{\infty} $ and $\lim_{x_0\to\infty}P_t f(x_0)=0$ uniformly, we know that $ P $ acts as a semigroup of contractions on $ \mathrm{C}_0(\mathbb{R}^N) $.
	Let us introduce a splitting, i.e.~two semigroups $ P^1 $ and $ P^2 $ associated with the equations
	\begin{equation}
		\dd z^1(t,x_0) = V_0 (z^1(t,x_0)) \dd t
	\end{equation}
	and
	\begin{equation}
		\dd z^2(t,x_0) = V_1 (z^2(t,x_0)) \circ\dd B_t,
	\end{equation}
	where $ V_0 $ is the Stratonovich corrected drift term. 
	Apparently, these two semigroups are contractions, too, and it remains to show that we have a short time asymptotic expansion on some subspace $ \mathcal{M}\subset\mathrm{C}_0(\mathbb{R}^n) $. For $ Q_{(\Delta t)}:= P^1_{\Delta t} P^2_{\Delta t} $ we can choose any $\mathrm{C}^3$-function $f $, which is bounded with compact support, and we obtain by It\^o's formula
	$$
	\lVert P_{\Delta t} P_s f - Q_{(\Delta t)} P_s f \rVert_{\infty} \leq C_f \Delta t^2
	$$
	for $ \Delta t $ in some small interval $ [0,\varepsilon] $. 
	The previous result then leads to the desired convergence, which has the well-known meaning of weak convergence of the associated processes to $ x(t,x_0) $. 
	However, two questions remain at this point:
	is it possible to obtain the convergence also for functions, which are not compactly supported, or not even globally bounded?
	If we want to relax towards $ f \notin\mathrm{C}_0(\mathbb{R}^N) $, we have to give up linear growth of vector fields and replace it by boundedness. 
	This raises the important question: is it possible to obtain rates of convergence in a generic setting for unbounded, non-compactly supported payoffs $ f $ and vector fields with linear growth? 

	The answer to the first part of this question will also answer the second part.
	We introduce a weight function $ \psi : \mathbb{R}^N \to (0,\infty) $ such that
	\begin{equation}
		\exp(- \alpha t ) \psi(x(t,x_0))
	\end{equation}
	is a supermartingale for every $ x_0 \in \mathbb{R}^N $. 
	We can easily choose such weight functions, even if the vector fields are linearly growing, as polynomials, and we can do so simultaneously for $ x , z^1 , z^2 $. 
	We need the uniform bound on moments of diffusions with linearly growing vector fields and It\^o's formula. Then we consider the Banach spaces $ \mathcal{B}^{\psi} (X) $ of those functions which can be approximated by bounded continuous functions with respect to the norm 
	$$ 
	\lVert f \rVert_{\psi} := \sup_{x \in \mathbb{R}^N} \frac{\lvert f(x)\rvert}{\psi(x)}.
	$$

	Apparently all semigroups are extending to this space and their respective norms are bounded by $ \exp(\alpha t) $. This finally yields that we are again in the assumptions of the previous meta-theorem, i.e.
	$$
	\lVert P_t f -  {(Q_{(\frac{t}{n})})}^n f \rVert_{\psi} \leq C_f \frac{1}{n}
	$$
	as $ n \to \infty $, for $ f \in \mathcal{M} $, which are $\mathrm{C}^2$-functions with appropriate boundedness relative to $\psi$. 
	Notice that we have extended the previous result twofold:
	in the present setting, both linearly growing volatility vector fields and linearly growing payoffs are allowed.
	The price to pay was that all results are with respect to a weighted supremum norm.
\end{example}

\section{Riesz Representation for Weighted Spaces}
In this section we show that we can actually obtain a variant of the Riesz representation theorem even on spaces that are not locally compact. Consider a completely regular Hausdorff topological space $X$ (i.e.~$\mathrm{T}_{3.5}$).
\begin{definition}
	A function $\psi\colon X\to(0,\infty)$ is called \emph{admissible weight function} if the sets $K_R:=\left\{ x\in X\colon \psi(x)\le R \right\}$ are compact for all $R>0$.
\end{definition}
Such a function $\psi$ is lower semicontinuous and bounded from below, and any such space $X$ is $\sigma$-compact due to $\bigcup_{n\in\mathbb{N}}K_n=X$. We call the pair $(X,\psi)$ a \emph{weighted space}.

Consider the vector space 
\begin{equation}
	\mathrm{B}^\psi(X;Z):=\left\{ f\colon X\to Z \colon \sup_{x\in X}\psi(x)^{-1}\lVert f(x)\rVert<\infty \right\}
\end{equation}
of $ Z $-valued functions $ f $, $Z$ a Banach space, equipped with the norm 
\begin{equation}
	\lVert f\rVert_{\psi}:=\sup_{x\in X}\psi(x)^{-1}\lVert f(x) \rVert,
\end{equation}
turning it into a Banach space itself. 
It is clear that $\mathrm{C}_b(X;Z)\subset\mathrm{B}^\psi(X;Z)$, where $\mathrm{C}_b(X;Z)$ denotes the space of continuous, bounded functions $f\colon X\to Z$, endowed with the norm $\lVert f\rVert_{\mathrm{C}_b(X;Z)}:=\sup_{x\in X}\lVert f(x)\rVert$.
\begin{definition}
	We define $\mathcal{B}^{\psi}(X;Z)$ as the closure of $\mathrm{C}_b(X;Z)$ in $\mathrm{B}^{\psi}(X;Z)$. The normed space $\mathcal{B}^{\psi}(X;Z)$ is a Banach space.
\end{definition}
\begin{remark}
	Suppose $X$ compact. Then the choice $\psi(x) = 1$ for $ x \in X $ is admissible. On general spaces weights $ \psi $ necessarily grow due to the compactness of $K_R$, which means that $f\in\mathcal{B}^{\psi}(X;Z)$ typically is unbounded, but its growth is bounded by the growth of $\psi$.
	Therefore, we call elements of $\mathcal{B}^{\psi}(X;Z)$ \emph{functions with growth controlled by $ \psi$}.
\end{remark}
We set $\mathcal{B}^\psi(X):=\mathcal{B}^\psi(X;\mathbb{R})$.
\begin{theorem}[Riesz representation for $\mathcal{B}^\psi(X)$]\label{theorem:rieszrepresentation} 
	Let $\ell\colon\mathcal{B}^\psi(X)\to\mathbb{R}$ be a continuous linear functional.
	Then, there exists a finite signed Radon measure $\mu$ on $X$ such that
	\begin{equation}
		\ell(f)=\int_{X}f(x)\mu(\dd x)\quad\text{for all $f\in\mathcal{B}^\psi(X)$.}
	\end{equation}
	Furthermore,
	\begin{equation}
		\label{eq:rieszrepresentation-psiintbound}
		\int_{X}\psi(x)\lvert \mu\rvert(\dd x) = \lVert \ell\rVert_{L(\mathcal{B}^\psi(X),\mathbb{R})},
	\end{equation}
	where $\lvert\mu\rvert$ denotes the total variation measure of $\mu$.
\end{theorem}
As every such measure defines a continuous linear functional on $\mathcal{B}^\psi(X)$, this completely characterises the dual space of $\mathcal{B}^\psi(X)$.
\begin{proof}
	Clearly, $\ell|_{\mathrm{C}_b(X)}$ is a continuous linear functional on $\mathrm{C}_b(X)$, as
	\begin{equation}
		\lVert f\rVert_{\psi} \le \left( \inf_{x\in X}\psi(x) \right)^{-1}\lVert f\rVert_{\mathrm{C}_b(X)} \quad\text{for $f\in\mathrm{C}_b(X)$}.
	\end{equation}

	We thus have to ensure condition \emph{(M)} of \cite[\S~5 Proposition~5]{Bourbaki1969}. Defining $K:=K_{\varepsilon^{-1}\lVert \ell \rVert_{L(\mathcal{B}^\psi(X),\mathbb{R})}}$, we see that for $g\in\mathrm{C}_b(X)$ with $\lvert g\rvert\le 1$ and $g|_{K}=0$,
	\begin{equation}
		\lVert g\rVert_{\psi} = \sup_{x\in X\setminus K}\psi(x)^{-1}\lvert g(x)\rvert \le \varepsilon \lVert \ell\rVert_{L(\mathcal{B}^{\psi}(X),\mathbb{R})}^{-1}\lVert g\rVert_{\mathrm{C}_b(X)} \le \varepsilon \lVert \ell\rVert_{L(\mathcal{B}^{\psi}(X),\mathbb{R})}^{-1},
	\end{equation}
	and thus $\lvert \ell(g)\rvert \le \varepsilon$. Hence we obtain existence of a finite, uniquely determined signed Radon measure $\mu$ with $\ell(f)=\int_{X}f(x)\mu(\dd x)$ for all $f\in\mathrm{C}_b(X)$ (see also \cite[Chapter~2 Theorem~2.2]{BergChristensenRessel1984}).

	To determine $\int_X \psi(x)\lvert \mu\rvert(\dd x)$, we apply \cite[\S~5 Proposition~1b)]{Bourbaki1969}:
	$\psi$ is lower semicontinuous and every positive $g\in\mathrm{C}_b(X)$ with $g\le\psi$ satisfies $\lVert g\rVert_{\psi}\le 1$. 
	Therefore,
	\begin{equation}
		\int_{X}\psi(x)\lvert\mu\rvert(\dd x) = \sup_{\substack{g\in\mathrm{C}_b(X)\\ \lvert g\rvert\le\psi}}\lvert \ell(g)\rvert \le \lVert \ell\rVert_{L(\mathcal{B}^\psi(X),\mathbb{R})}.
	\end{equation}
	The density of $\mathrm{C}_b(X)$ in $\mathcal{B}^\psi(X)$ yields
	\begin{alignat}{2}
		\lVert \ell\rVert_{L(\mathcal{B}^\psi(X);\mathbb{R})}
		&= \sup_{g\in\mathrm{C}_b(X)}\lVert g\rVert_{\psi}^{-1}\lvert \ell(g)\rvert 
		= \sup_{g\in\mathrm{C}_b(X)}\lVert g\rVert_{\psi}^{-1}\Bigl\lvert \int_X g(x)\mu(\dd x)\Bigr\rvert \notag \\
		&\le
		\int_{X}\psi(x)\lvert\mu\rvert(\dd x).
	\end{alignat}
	Hence, $\int_{X}\psi(x)\lvert\mu\rvert(\dd x)=\lVert\ell\rVert_{L(\mathcal{B}^\psi(X);\mathbb{R})}$.

	For the proof of $\ell(f)=\int_{X}f(x)\mu(\dd x)$ for all $f\in\mathcal{B}^\psi(X)$, note that $f\mapsto\int_{X}f(x)\mu(\dd x)$ defines a continuous linear functional on $\mathcal{B}^\psi(X)$ due to the integrability of $\psi$ with respect to $\lvert\mu\rvert$. 
	As both expressions agree on a dense subset, we obtain the desired equality.
\qed\end{proof}
\begin{remark}
	While the result in \cite[Chapter~2 Theorem~2.2]{BergChristensenRessel1984} is more general, we do not see how to prove $\int_X \psi(x)\lvert\mu\rvert(\dd x)<\infty$ in that situation. 
	However, this bound is important in our further results, see the proof of Theorem \ref{theorem:Ttstrongcont}.
\end{remark}

\begin{corollary}
	Let $\ell\colon\mathcal{B}^{\psi}(X)\to\mathbb{R}$ be a positive linear functional, that is, $\ell(f)\ge 0$ whenever $f(x)\ge 0$ for all $x\in X$. 
	Then, there exists a (positive) measure $\mu$ with $\ell(f)=\int_{X}f(x)\mu(\dd x)$ for every $f\in\mathcal{B}^\psi(X)$.
\end{corollary}
\begin{proof}
	We only have to prove that $\ell$ is continuous.
	Assume otherwise. 
	Then, there exists a sequence $(f_n)_{n\in\mathbb{N}}$, $f_n\in\mathcal{B}^{\psi}(X)$, such that $\lVert f_n\rVert_{\psi}=1$, but $\lvert \ell(f_n)\rvert\ge n^3$. As $\lvert \ell(f)\rvert\le \ell(\lvert f\rvert)$ for any $f\in\mathcal{B}^{\psi}(X)$, we can assume without loss of generality that $f_n\ge 0$ for all $n\in\mathbb{N}$. 
	As $\sum_{n\in\mathbb{N}}n^{-2}\lVert f_n\rVert_{\psi}<\infty$, the limit $f:=\sum_{n\in\mathbb{N}}n^{-2}f_n\in\mathcal{B}^{\psi}(X)$ is well-defined and $f\ge 0$. 
	Thus, we obtain a contradiction due to
	\begin{equation}
		n \le \ell(n^{-2}f_n) \le \ell(f). 
	\end{equation}
\qed\end{proof}

The following results emphasise the analogy in structure of $\mathcal{B}^\psi(X)$ and the space of functions vanishing at infinity on a locally compact space.
\begin{theorem}
	\label{theorem:Bdecay}
	Let $f\colon X\to\mathbb{R}$.
	Then, $f\in\mathcal{B}^{\psi}(X)$ if and only if $f|_{K_R}\in \mathrm{C}(K_R)$ for all $R>0$ and
	\begin{equation}
		\label{eq:Bdecay}
		\lim_{R\to\infty}\sup_{x\in X\setminus K_R}\psi(x)^{-1}\lvert f(x)\rvert=0.
	\end{equation}
	In particular, $f\in\mathcal{B}^\psi(X)$ for every $f\in\mathrm{C}(X)$ where \eqref{eq:Bdecay} holds.
\end{theorem}
\begin{proof}
	Assume that $f\in\mathcal{B}^\psi(X)$.
	For $g\in\mathrm{C}_b(X)$ with $\lVert f-g\rVert_{\psi}<\frac{\varepsilon}{2}$,
	\begin{alignat}{2}{}
		\psi(x)^{-1}\lvert f(x)\rvert \le \frac{\varepsilon}{2} + \psi(x)^{-1}\rvert g(x)\rvert \quad\text{for $x\in X$},
	\end{alignat}
	the last term being bounded by $\frac{\varepsilon}{2}$ for $x\in X\setminus K_{R}$ with $R:=2\varepsilon^{-1}\lVert g\rVert_{\mathrm{C}_b(X)}$.
	Thus,
	\begin{equation}
		\sup_{x\in X\setminus K_R}\psi(x)^{-1}\lvert f(x)\rvert
		\le
		\varepsilon,
	\end{equation}
	which proves \eqref{eq:Bdecay}.

	Next, we prove that for any $R>0$, $f|_{K_R}$ is continuous.
	With $g$ as above,
	\begin{equation}
		\sup_{x\in K_R}\lvert f(x)-g(x)\rvert
		\le
		R\sup_{x\in K_R}\psi(x)^{-1}\lvert f(x)-g(x)\rvert
		\le
		\frac{\varepsilon}{2} R,
	\end{equation}
	which means that $f|_{K_R}$ is a uniform limit of continuous functions and hence continuous.

	For the other direction, set $f_n:=\min(\max(f(\cdot),-n),n) = (f_n \vee n) \wedge n $.
	We prove first that $f_n\in\mathcal{B}^\psi(X)$.
	As $f|_{K_R}\in\mathrm{C}(K_R)$, we see that $f_n|_{K_R}\in\mathrm{C}(K_R)$.
	$K_R$ is compact in a completely regular space.
	We can embed $X$ into a compact space $Y$ by \cite[Chapitre~IX~\S~1~Proposition~3, Proposition~4]{Bourbaki1974}.
	Applying the Tietze extension theorem \cite[Chapitre~IX~\S~4~Th\'eor\`eme~2]{Bourbaki1974} to the set $K_R$, which is also compact in $Y$ and therefore closed, we obtain existence of $g_{n,R}\in\mathrm{C}_b(X)$ with $g_{n,R}|_{K_R}=f_{n}|_{K_R}$ and $\sup_{x\in X}\lvert g_{n,R}(x)\rvert\le n$ for all $x\in X$.
	\eqref{eq:Bdecay} yields
	\begin{equation}
		\lVert f_n - g_{n,R}\rVert_{\psi}
		\le \sup_{x\in X\setminus K_R}\psi(x)^{-1}\lvert f_n(x)-g_{n,R}(x)\rvert
		\le 2nR^{-1},
	\end{equation}
	hence $f_n\in\mathcal{B}^{\psi}(X)$.
	Next, choose $R>0$ such that $\sup_{x\in X\setminus K_R}\psi(x)^{-1}\lvert f(x)\rvert<\varepsilon$.
	With $n>\sup_{x\in K_R}\lvert f(x)\rvert$, $f(x)=f_n(x)$ on $K_R$.
	Therefore,
	\begin{equation}
		\lVert f-f_n\rVert_{\psi}
		\le \sup_{x\in X\setminus K_R}\psi(x)^{-1}\lvert f(x)-f_n(x)\rvert
		\le 2\varepsilon,
	\end{equation}
	which shows that $f\in\mathcal{B}^\psi(X)$.
\qed\end{proof}
\begin{theorem}
	\label{theorem:attainmax}
	For every $f\in\mathcal{B}^\psi(X)$ with $\sup_{x\in X}f(x)>0$, there exists $z\in X$ such that
	\begin{equation}
		\psi(x)^{-1}f(x) \le \psi(z)^{-1}f(z) \quad\text{for all $x\in X$}.
	\end{equation}
\end{theorem}
\begin{proof}
	Let $\alpha:=\sup_{x\in X}\psi(x)^{-1}f(x)>0$.
	Then, by Theorem~\ref{theorem:Bdecay}, there exists some $R>0$ such that $\sup_{\psi(x)>R}\psi(x)^{-1}f(x)\le\frac{\alpha}{2}$, that is, 
	\begin{equation}
		\alpha=\sup_{x\in K_R}\psi(x)^{-1}f(x).
	\end{equation}
	Define $h:=\psi^{-1}\max(f,0)$.
	Then, $\alpha=\sup_{x\in K_R}h(x)$.
	Furthermore, $\psi^{-1}$ is upper semicontinuous, $\max(f,0)$ is continuous on $K_R$ by Theorem~\ref{theorem:Bdecay} and both are nonnegative.
	Thus, $h$ is upper semicontinuous (see \cite[Chap.~IV~\S~6~Proposition~2]{Bourbaki1971}) and by \cite[Chapitre~IV~\S~6~Th\'eor\`eme~3]{Bourbaki1971} attains its maximum at some point $z\in K_R$, i.e., $\alpha=\psi(z)^{-1}f(z)$
\qed\end{proof}

\section{A Generalised Feller Condition}
The generalised Feller property will allow us to speak about strongly continuous semigroups on spaces of functions with growth controlled by $ \psi $, in particular functions which are unbounded. From the point of view of applications this means that we consider a weighted supremum norm instead of the supremum norm.

Let $(P_t)_{t\ge 0}$ be a family of bounded linear operators $P_t\colon\mathcal{B}^{\psi}(X)\to\mathcal{B}^{\psi}(X)$ with the following properties:
\begin{enumerate}
		\renewcommand{\theenumi}{{\bf F\arabic{enumi}}}
	\item
		\label{enu:defgenfeller-0id}
		$P_0=I$, the identity on $\mathcal{B}^{\psi}(X)$,
	\item
		$P_{t+s}=P_tP_s$ for all $t$, $s\ge 0$,
	\item
		\label{enu:defgenfeller-pwconv}
		for all $f\in\mathcal{B}^{\psi}(X)$ and $x\in X$, $\lim_{t\to 0+}P_t f(x)=f(x)$,
	\item
		\label{enu:defgenfeller-bound}
		there exist a constant $C\in\mathbb{R}$ and $\varepsilon>0$ such that for all $t\in [0,\varepsilon]$, $\lVert P_t\rVert_{L(\mathcal{B}^{\psi}(X))}\le C$,
	\item
		\label{enu:defgenfeller-positivity}
		$P_t$ is positive for all $t\ge 0$, that is, for $f\in\mathcal{B}^{\psi}(X)$, $f\ge 0$, we have $P_t f\ge 0$.
\end{enumerate}
Alluding to \cite[Chapter~17]{Kallenberg1997}, such a family of operators will be called a \emph{generalised Feller semigroup}.

\begin{remark}
	As Chris Rogers remarked, state space transformation of the type $ x \mapsto \phi(x):= \frac{x}{\sqrt{1+{\| x \|}^2}} $ transform unbounded state spaces into bounded ones. The weight function $ \psi $ is then used to rescale real valued functions $ f : X \to \mathbb{R} $ via $ \tilde{f}:=f/ \psi $ in order to investigate $ \tilde{f} \circ \phi^{-1} $ on $ \phi(X) $. This function will often have a continuous extension to the closure of $ \phi(X) $, which -- in the appropriate topology -- will be often compact. This relates the generalized Feller property to the classical Feller property.
	Note that in our situation, however, $\psi$ is typically not continuous for infinite dimensional $X$.
\end{remark}

We shall now prove that semigroups satisfying \ref{enu:defgenfeller-0id} to \ref{enu:defgenfeller-bound} are actually strongly continuous, a direct consequence of Lebesgue's dominated convergence theorem with respect to measures existing due to Riesz representation.
\begin{theorem}
	\label{theorem:Ttstrongcont}
	Let $(P_t)_{t\ge 0}$ satisfy \ref{enu:defgenfeller-0id} to \ref{enu:defgenfeller-bound}. 
	Then, $(P_t)_{t\ge 0}$ is strongly continuous on $\mathcal{B}^{\psi}(X)$, that is,
	\begin{equation}
		\lim_{t\to 0+}\lVert P_t f-f\rVert_{\psi}=0
		\quad\text{for all $f\in\mathcal{B}^{\psi}(X)$}.
	\end{equation}
\end{theorem}
\begin{proof}
	By \cite[Theorem I.5.8]{EngelNagel2000}, we only have to prove that $t\mapsto \ell(P_t f)$ is right continuous at zero for every $f\in\mathcal{B}^{\psi}(X)$ and every continuous linear functional $\ell\colon\mathcal{B}^{\psi}(X)\to\mathbb{R}$.
	Due to Theorem~\ref{theorem:rieszrepresentation}, we know that there exists a signed measure $\nu$ on $X$ such that $\ell(g)=\int_{X}g\dd\nu$ for every $g\in\mathcal{B}^{\psi}(X)$.
	By \ref{enu:defgenfeller-bound}, we see that for every $t\in[0,\varepsilon]$,
	\begin{equation}
		\lvert P_t f(x)\rvert
		\le C \psi(x).
	\end{equation}
	Due to \eqref{eq:rieszrepresentation-psiintbound}, the dominated convergence theorem yields
	\begin{alignat}{2}{}
		\lim_{t\to 0+}\int_{X}P_t f(x)\nu(\dd x) = \int_{X}f(x)\nu(\dd x),
	\end{alignat}
	and the claim follows.
	Here, the integrability of $ \psi $ with respect to the total variation measure $ \lvert\nu\rvert $ enters in an essential way.
\qed\end{proof}

We can establish a positive maximum principle in case that the semigroup $ P_t $ grows like $ \exp(\alpha t) $ with respect to the operator norm on $ \mathcal{B}^{\psi}(X) $.
\begin{theorem}
	\label{theorem:Ttposmaxprinciple}
	Let $A$ be an operator on $\mathcal{B}^{\psi}(X)$ with domain $D$, and $\omega\in\mathbb{R}$.
	$A$ is closable with its closure $\overline{A}$ generating a generalised Feller semigroup $(P_t)_{t\ge 0}$ with $\lVert P_t\rVert_{L(\mathcal{B}^{\psi}(X))}\le\exp(\omega t)$ for all $t\ge 0$ if and only if 
	\begin{enumerate}
		\item
			$D$ is dense,
		\item
			$A-\lambda_0$ has dense image for some $\lambda_0>\omega$, and
		\item
			$A$ satisfies the \emph{generalised positive maximum principle}, that is, for $f\in D$ with $(\psi^{-1}f)\vee 0\le \psi(z)^{-1}f(z)$ for some $z\in X$, $Af(z)\le \omega f(z)$.
	\end{enumerate}
\end{theorem}
Note that $(\psi^{-1}f)\vee 0=\psi^{-1}(f\vee 0)$ as $\psi>0$.
Therefore, $(\psi^{-1}f)\vee 0\le\psi^{-1}(z)f(z)$ is equivalent to
\begin{equation}
	\lVert f\vee 0\rVert_{\psi} 
	\le \psi^{-1}(z)f(z).
\end{equation}
\begin{proof}
	Assume first that $(P_t)_{t\ge 0}$ is a generalised Feller semigroup satisfying 
	\begin{equation}
		\lVert P_t\rVert_{L(\mathcal{B}^{\psi}(X))}\le\exp(\omega t),
	\end{equation}
	and $A$ with domain $D$ is its generator.
	For $f\in D$ with $\lVert f\vee 0\rVert_{\psi}\le\psi^{-1}(z)f(z)$,
	\begin{alignat}{2}{}
		P_t f(z) 
		&\le P_t (f\vee 0)(z) \le \psi(z)\lVert P_t(f\vee 0)\rVert_{\psi} \le \psi(z)\exp(\omega t)\lVert f\vee 0\rVert_{\psi} \notag \\
		&\le\exp(\omega t)f(z),
	\end{alignat}
	and due to the continuity of point evaluation, we obtain the inequality $Af(z)\le \omega f(z)$ in the limit $t\to 0+$.
	Thus, $A$ satisfies the generalised positive maximum principle.
	The density of $D$ and of $(A-\lambda_0)D$ follow from the Hille-Yosida theorem \cite[Theorem II.3.8, p.~77]{EngelNagel2000}.

	For the other direction, let $f\in D$ be arbitrary, and define $g:=(\sgn f(z))f$, where $z$ is chosen such that $\psi(z)^{-1}\lvert f(z)\rvert=\lVert f\rVert_{\psi}$ (this is possible due to Theorem~\ref{theorem:attainmax}).
	Clearly, $g\in D$ and $\psi(x)^{-1}g(x)\le\psi(z)^{-1}g(z)$, so the generalised positive maximum principle yields $Ag(z)\le \omega g(z)$.
	Thus, for $\lambda>0$,
	\begin{alignat}{2}
		\lVert (\lambda-(A-\omega))f\rVert_{\psi}
		&\ge
		\psi(z)^{-1}\left( \lambda g(z)-(A-\omega)g(z) \right)
		\ge
		\psi(z)^{-1}\lambda g(z) \notag \\
		&=
		\lambda\lVert f\rVert_{\psi}.
	\end{alignat}
	From this, closability of $A$ follows:
	if $(f_n)_{n\in\mathbb{N}}$ in $D$ are given such that both $\lim_{n\to\infty}\lVert f_n\rVert_{\psi}=0$ and $\lim_{n\to\infty}\lVert Af_n-g\rVert_{\psi}=0$, there exist $(g_m)_{m\in\mathbb{N}}$ in $D$ with $\lim_{m\to\infty}\lVert g_m-g\rVert_{\psi}=0$.
	Thus, for any $\lambda>0$ and $m$, $n\in\mathbb{N}$,
	\begin{equation}
		\lVert (\lambda-(A-\omega))(g_m+\lambda f_n)\rVert_{\psi}
		\ge \lambda \lVert g_m+\lambda f_n\rVert.
	\end{equation}
	Taking the limit $n\to\infty$, dividing by $\lambda$ and taking the limit $\lambda\to\infty$, we obtain $\lVert g_m-g\rVert_{\psi}\ge \lVert g_m\rVert_{\psi}$, and the limit $m\to\infty$ yields $g=0$.
	This proves the closability of $A$, and the closure $\overline{A}$ of $A$ with domain $\mathcal{D}$ satisfies
	\begin{equation}
		\lVert (\lambda-(\overline{A}-\omega))f\rVert_{\psi}
		\ge
		\lambda\lVert f\rVert_{\psi}
		\quad\text{for all $\lambda>0$ and $f\in\mathcal{D}$}.
	\end{equation}
	Thus, $\overline{A}-\omega$ is dissipative.
	The Lumer-Phillips theorem \cite[Theorem~II.3.15]{EngelNagel2000} yields that $\overline{A}$ generates a semigroup with $\lVert P_t\rVert_{L(\mathcal{B}^{\psi}(X))}\le\exp(\omega t)$ for all $t\ge 0$.

	We now prove that $R_\lambda=(\lambda-\overline{A})^{-1}$ is positive for every $\lambda>\omega$, which yields that $P_t$ is positive for every $t\ge 0$.
	To this end, we show that given $g\in\mathcal{B}^{\psi}(X)$ such that the solution $f\in\mathcal{D}$ of $(\lambda-\overline{A})f=g$ is not positive, $g$ cannot be positive, either.
	By assumption, $\alpha:=\inf_{x\in X}\psi(x)^{-1}f(x)<0$.
	Given a sequence of functions $(f_n)_{n\in\mathbb{N}}$ in $D$ converging to $f$ such that $Af_n$ converges to $\overline{A}f$, we see that we can assume without loss of generality that for every $n\in\mathbb{N}$, $\alpha_n:=\inf_{x\in X}\psi(x)^{-1}f_n(x)<0$, and we have that $\lim_{n\to\infty}\alpha_n=\alpha$.
	Theorem~\ref{theorem:attainmax} yields the existence of $z_n\in X$ with $\psi(z_n)^{-1}f_n(z_n)=\alpha_n$.
	By the positive maximum principle, $Af_n(z_n)\ge \omega f_n(z_n)$.
	Thus,
	\begin{alignat}{2}{}
		\inf_{x\in X}\psi(x)^{-1}g(x) 
		&= \lim_{n\to\infty} \inf_{x\in X}\psi(x)^{-1}(\lambda-A)f_n(x) \notag \\ 
		&\le \lim_{n\to\infty} \psi(z_n)^{-1}(\lambda-A)f_n(z_n) \notag \\ 
		&\le \lim_{n\to\infty} \psi(z_n)^{-1}(\lambda-\omega)f_n(z_n) \notag \\ 
		&= (\lambda-\omega) \lim_{n\to\infty} \inf_{x\in X} \psi(x)^{-1}f_n(x) \notag \\
		&= (\lambda-\omega) \inf_{x\in X} \psi(x)^{-1}f(x) = (\lambda-\omega)\alpha < 0,
	\end{alignat}
	that is, $g$ is not positive.
\qed\end{proof}

\section{Results On Dual Spaces}
In this section we consider a special class of state spaces that will be crucial for our applications to SPDEs: dual spaces of Banach spaces equipped with the weak-$*$ topology.
We remark that the weak topology on Hilbert spaces and sequential weak continuity was also used by Maslowski and Seidler \cite{MaslowskiSeidler1999} to prove ergodicity of stochastic partial differential equations.

Assume that $X=Y^*$ is the dual space of some Banach space $Y$ with its weak-$*$ topology or, more generally, a Hausdorff topological vector space.
Such a space is clearly endowed with a uniform structure, and thus completely regular Hausdorff \cite[Chapitre~IX~\S~1~Th\'eor\`eme~2]{Bourbaki1974}.
Consider a lower semicontinuous function $\psi\colon X\to(0,\infty)$.
Compactness of $K_R$ can often be proved using the Banach-Alaoglu theorem \cite[Theorem 3.15]{Rudin1973}.
In particular, if $Y$ is a Banach space and the sets $K_R$ are bounded in norm in $X$, compactness follows.

We denote by $X_{w*}$ the space $X$ endowed with the weak-$*$ topology, and assume that $(X_{w*},\psi)$ is a weighted space for a given weight function.
The sets 
$$
K_R=\left\{ x\in X\colon \psi(x)\le R \right\}
$$ 
are then weak-$*$ compact, and we shall always consider the weak-$*$ topology on $K_R$.
\begin{example}
	\label{ex:weightsBanachspace}
	Typical examples for weight functions are of the form $\psi(x)=\rho(\lVert x\rVert)$, where $\rho\colon[0,\infty)\to(0,\infty)$ is increasing and left-continuous.
	In this case, 
	\begin{equation}
		K_R=C_r(0):=\left\{ x\in X\colon \lVert x\rVert\le r \right\},
	\end{equation}
	where $r=\max\left\{ p\in\mathbb{R}\colon \rho(p)\le R \right\}$, and $C_r(0)$ is weak-$*$ compact by the Ba\-nach-Alaoglu theorem.
	Note that $\rho(r)\le R$ by left continuity.
	Below, we will consider choices such as $\rho(t)=(1+t^2)^{s/2}$, $s\ge 2$, $\rho(t)=\cosh(\beta t)$, $\beta>0$, and $\rho(t)=\exp(\eta t^2)$, $\eta>0$.
\end{example}
We want to give an approximation result for functions in $\mathcal{B}^\psi(X_{w*})$ by cylindrical functions.
Set
\begin{alignat}{2}
	\mathcal{A}_N := \bigl\{ g(\langle\cdot,y_1\rangle,\dots,\langle\cdot,y_N\rangle)\colon 
	&\text{$g\in\mathrm{C}_b^{\infty}(\mathbb{R}^N)$} \notag
	\\
	&\text{and $y_j\in Y$, $j=1,\dots,N$} \bigr\},
\end{alignat}
and denote by $\mathcal{A}:=\bigcup_{N\in\mathbb{N}}\mathcal{A}_N$ the bounded smooth continuous cylinder functions on $X$.
Clearly, $\mathcal{A}\subset\mathcal{B}^\psi(X_{w*})$.
\begin{theorem}
	\label{theorem:boundedweakcontapprox}
	The closure of $\mathcal{A}$ in $\mathrm{B}^\psi(X_{w*})$ coincides with $\mathcal{B}^\psi(X_{w*})$.
\end{theorem}
\begin{proof}
	We prove first by the Stone-Weierstrass theorem \cite{Rudin1973} that $\mathcal{A}$ is dense in $\mathrm{C}_b(K_R)$ for any $R>0$. 
	First, it is obvious that $\mathcal{A}$ is an algebra, as $\mathcal{A}_{N}\cdot\mathcal{A}_{M}\subset\mathcal{A}_{N+M}$ for all $N$ and $M$ with obvious notation, and $\mathcal{A}_{N}\subset\mathcal{A}_{N+1}$ for all $N\in\mathbb{N}$.
	Moreover, for any $x_1\ne x_2$, $x_1$, $x_2\in K_R$, there exists some $y\in Y$ with $\langle x_1,y\rangle\ne\langle x_2,y\rangle$, which clearly yields that already $\mathcal{A}_1$ separates points.
	As the constant functions are obviously in $\mathcal{A}$, we obtain density in $\mathrm{C}_b(K_R)$.

	Let now $f\in\mathrm{C}_b(X_{w*})$.
	Then, for every $R>0$ and $\varepsilon>0$, there exists some $N\in\mathbb{N}$ and $\tilde{f}_{R,\varepsilon}\in\mathcal{A}_{N}\subset\mathcal{B}^{\psi}(X)$ with 
	\begin{equation}
		\sup_{x\in K_R}\lvert f(x)-\tilde{f}_{R,\varepsilon}(x)\rvert <\varepsilon.
	\end{equation}
	By definition, $\tilde{f}_{R,\varepsilon}=\tilde{g}\circ h$ with $h(x)=\left( \langle x,y_j\rangle \right)_{j=1,\dots,N}$ for some $y_j\in Y$ and $\tilde{g}\in\mathrm{C}_b^{\infty}(\mathbb{R}^N)$.
	As $K_R$ is compact, $h(K_R)\subset\mathbb{R}^N$ is compact.
	By the Tietze extension theorem \cite[Chapitre~IX~\S~4~Th\'eor\`eme~2]{Bourbaki1974}, we can extend $\tilde{g}|_{h(K_R)}$ to a continuous function $\hat{g}$ on $\mathbb{R}^N$ with $\sup_{y\in\mathbb{R}^N}\lvert\hat{g}(y)\rvert\le\sup_{x\in K_R}\lvert \tilde{f}_{R,\varepsilon}(x)\rvert$.
	Applying \cite[Proposition~IV.21, Proposition~IV.20]{Brezis1994}, we see that convolution of $\hat{g}$ with a mollifier yields a function $g\in\mathrm{C}_b^\infty(\mathbb{R}^N)$ with $\sup_{y\in\mathbb{R}^N}\lvert g(y)\rvert\le\sup_{x\in K_R}\lvert \tilde{f}_{R,\varepsilon}(x)\rvert$ and $\sup_{y\in h(K_R)}\lvert g(y)-\tilde{g}(y)\rvert<\varepsilon$.
	Assuming without loss of generality that 
	\begin{equation}
		\sup_{x\in K_R}\lvert \tilde{f}_{R,\varepsilon}(x)\rvert \le 2\sup_{x\in K_R}\lvert f(x)\rvert,
	\end{equation}
	we see that $f_{R,\varepsilon}:=g\circ h$ satisfies
	\begin{equation}
		\sup_{x\in K_R}\lvert f(x)-f_{R,\varepsilon}(x)\rvert < 2\varepsilon
		\quad\text{and}\quad
		\sup_{x\in X}\lvert f_{R,\varepsilon}(x)\rvert \le 2 \sup_{x\in X}\lvert f(x)\rvert,
	\end{equation}
	independently of $R$ and $\varepsilon$.
	Therefore, as $\psi(x)\ge \delta$ for all $x\in X$,
	\begin{alignat}{2}{}
		\lVert f-f_{R,\varepsilon}\rVert_{\psi}
		&\le
		\sup_{x\in K_R}\psi(x)^{-1}\lvert f(x)-f_{R,\varepsilon}(x)\rvert + \sup_{\psi(x)>R}\psi(x)^{-1}\lvert f(x)-f_{R,\varepsilon}(x)\rvert \notag \\
		&\le
		\delta^{-1}\sup_{x\in K_R}\lvert f(x)-f_{R,\varepsilon}(x)\rvert + 3R^{-1}\sup_{x\in X}\lvert f(x)\rvert.
	\end{alignat}
	The result follows.
\qed\end{proof}
The definition of $\mathcal{A}$ is not ``optimal'' in the sense that it will contain too many functions.
The following result is significantly better in this respect.
\begin{theorem}
	\label{theorem:boundedweakcontapproxsep}
	Assume that $Y$ is separable, and let $\left\{ y_j\colon j\in\mathbb{N} \right\}\subset Y$ be a countable set which separates the points of $X=Y^*$.
	Define
	\begin{equation}
		\widetilde{\mathcal{A}}_N
		:=
		\left\{ g(\langle\cdot,y_1\rangle,\dots,\langle\cdot,y_N\rangle) \colon g\in\mathrm{C}_b^{\infty}(\mathbb{R}^N) \right\},
	\end{equation}
	and $\widetilde{\mathcal{A}}:=\bigcup_{N\in\mathbb{N}}\widetilde{A}_N\subset\mathcal{B}^\psi(X_{w*})$.
	Then, $\widetilde{\mathcal{A}}$ is dense in $\mathcal{B}^\psi(X_{w*})$.
\end{theorem}
\begin{proof}
	The proof is done in the same way as for Theorem~\ref{theorem:boundedweakcontapprox}, using that for any $x_1$, $x_2\in X$ with $x_1\ne x_2$, there exists some $j\in\mathbb{N}$ with $\langle x_1,y_j\rangle\ne\langle x_2,y_j\rangle$.
\qed\end{proof}

\begin{remark}
	\label{rem:boundedweakcontapproxsep}
	A possible choice for $\left\{ y_j\colon j\in\mathbb{N} \right\}$ is given by any countable dense set in $Y$.
	In particular, the specific choice of the $y_j$ does not make any difference, which was also observed in \cite[Remark 5.9]{HairerMattingly2008}.
\end{remark}

\begin{lemma}
	\label{lem:Bcharseqwstarcont}
	Assume that $X=Y^{*}$ with $Y$ separable.
	\begin{enumerate}
		\item 
			$f\in\mathcal{B}^{\psi}(X_{w*})$ if and only if $f$ satisfies \eqref{eq:Bdecay} and $f|_{K_R}$ is sequentially weak-$*$ continuous for any $R>0$.
		\item
			If for every $r>0$ there exists some $R>0$ such that $C_r(0)\subset K_R$, then every $f\in\mathcal{B}^{\psi}(X_{w*})$ is sequentially weak-$*$ continuous.
			In particular, in this case, $\mathcal{B}^{\psi}\subset\mathrm{C}(X_{w*})$.
	\end{enumerate}
\end{lemma}
\begin{remark}
	The condition $C_r(0)\subset K_R$ is quite natural, and is satisfied by the choice $\psi(x)=\rho(\lVert x\rVert)$ with $\rho$ increasing and left-continuous from Example~\ref{ex:weightsBanachspace}.
	It is, however, not automatically satisfied, as the example $X=\mathbb{R}$, $\psi(x):=x^2+x^{-1}\chi_{(0,\infty)}$ shows.
	Here, $\chi_A(x) := 1$ for $x\in A$ and $0$ for $x\notin A$ denotes the indicator function of the set $A$.
	In this example, the conclusion of the second part of the above Theorem even fails, as is easily seen.
\end{remark}
\begin{proof}
	By Theorem~\ref{theorem:Bdecay}, we only have to equate sequential weak-$*$ and weak-$*$ continuity of $f|_{K_R}$ for any $R>0$.
	By compactness, $K_R$ is bounded by the Banach-Steinhaus theorem \cite[Th\'eor\`eme~II.1]{Brezis1994}, as for any $y\in Y$, 
	\begin{equation}
		\sup_{x\in K_R}\lvert \langle x,y\rangle\rvert<\infty.
	\end{equation}
	Thus, \cite[Th\'eor\`eme~III.25]{Brezis1994} shows that $K_R$ is metrisable, which means that weak-$*$ continuity and sequential weak-$*$ continuity coincide.
	Therefore, any function $f$ is sequentially weak-$*$ continuous if and only if it is weak-$*$ continuous on $K_R$, and the first claim follows.

	For the second claim, note that any weak-$*$ converging sequence $(x_n)_{n\in\mathbb{N}}$ is bounded by the Banach-Steinhaus theorem.
	Thus, by assumption, $(x_n)_{n\in\mathbb{N}}$ stays in $K_R$ for some $R>0$, and the weak-$*$ continuity of $f|_{K_R}$ yields the result.
	Finally, every such $f$ is continuous with respect to the norm topology, as every norm convergent sequence converges weak-$*$, as well.
\qed\end{proof}

\section{Generalised Feller Semigroups and S(P)DEs}
Assume from now on that $X=Y^*$ with $Y$ a separable Banach space.
Let $\left\{ y_j\colon j\in\mathbb{N} \right\}\subset Y$ be a countable set which separates the points of $X$. Again, we write $X_{w*}$ for $X$ endowed with the weak-$*$ topology.

\begin{assumption}
Let $(x(t,x_0))_{t\ge 0}$ be a time homogeneous Markov process on some stochastic basis $(\Omega,\mathcal{F},\mathbb{P},(\mathcal{F}_t)_{t\ge 0})$ satisfying the usual conditions with values in $X$, started at $x_0\in X$. We assume that $(x(t,x_0))_{t\ge 0}$ has right continuous trajectories with respect to the weak-$*$ topology on $ X$.
\end{assumption}
We want to derive conditions on $(x(t,x_0))_{t\ge 0}$ such that its Markov semigroup $(P_t)_{t\ge 0}$, given by $P_t f(x_0):=\mathbb{E}\left[ f(x(t,x_0)) \right]$, is strongly continuous on the space $\mathcal{B}^{\psi}(X_{w*})$ for an appropriately chosen weight function $\psi$.

\begin{assumption}
Let $ (X,\psi) $ be a weighted space and $ x(t,x_0) $ a Markov process on $ X $. We assume
the existence of constants $C$ and $\varepsilon>0$ with
\begin{equation}
	\label{eq:markovpsibound}
	\mathbb{E}[\psi(x(t,x_0))]\le C\psi(x_0)
	\quad\text{for all $x_0\in X$ and $t\in[0,\varepsilon]$}.
\end{equation}
\end{assumption}
We prove first that inequality \eqref{eq:markovpsibound} is related to boundedness of the transition operator on $ \mathcal{B}^{\psi}(X_{w*}) $, and to some supermartingale property.
\begin{lemma}
	\label{lem:markovpsibound}
	Assume \eqref{eq:markovpsibound} for some $C$ and $\varepsilon>0$. 
	Then $ \lvert\mathbb{E}[f(x(t,x_0))]\rvert \leq C \psi(x_0) $ for all $ f \in \mathcal{B}^{\psi}(X_{w*}) $, $ x_0 \in X $ and $ t \in [0,\varepsilon]$. 

	Furthermore, the condition
	\begin{equation}
		\label{eq:markovpsibound_supermartingale}
		\mathbb{E}[\psi(x(t,x_0))]\le \exp(\omega t) \psi(x_0)
		\quad\text{for all $x_0\in X$ and $t\in[0,\varepsilon]$}.
	\end{equation}
	is equivalent to the property that the process $ \exp(- \omega t) \psi(x(t,x_0)) $ is a supermartingale in its own filtration, and this leads to 
	\begin{equation}
		\lvert \mathbb{E}[f(x(t,x_0))]\rvert \leq \exp(\omega t) \psi(x_0) 
		\quad\text{for $ x_0 \in X $ and $ t \geq 0$} 
	\end{equation}
	for all $ f \in \mathcal{B}^{\psi}(X_{w*}) $.
\end{lemma}

\begin{lemma}
	\label{lem:pwcontatzero}
	Assume \eqref{eq:markovpsibound} for some $C$ and $\varepsilon>0$. Then
	\begin{equation}
		\lim_{t\to 0+} \mathbb{E}[f(x(t,x_0))] = f(x_0)
		\quad
		\text{for any $f\in\mathcal{B}^{\psi}(X_{w*})$ and $x_0\in X$}.
	\end{equation}
\end{lemma}
\begin{proof}
	Denoting by $\chi_A$ the indicator function of the set $A$, we choose $R>\psi(x_0)$ and consider
	\begin{alignat}{2}{}
		\lvert \mathbb{E}\left[ f(x(t,x_0)) \right] - f(x_0) \rvert
		\le
		&\mathbb{E}\left[ \lvert f(x(t,x_0)) - f(x_0) \rvert \chi_{[\psi(x(t,x_0))\le R]} \right] \notag \\
		&+ \mathbb{E}\left[ \lvert f(x(t,x_0))\rvert \chi_{[\psi(x(t,x_0))>R]} \right] \notag \\
		&+ f(x_0)\mathbb{P}\left[ \psi(x(t,x_0))>R \right].
	\end{alignat}
	By the Markov inequality, 
	\begin{equation}
		\mathbb{P}\left[ \psi(x(t,x_0))>R \right]\le R^{-1}\mathbb{E}\left[ \psi(x(t,x_0)) \right]\le CR^{-1}\psi(x_0).
	\end{equation}
	Given $\varepsilon>0$, Theorem~\ref{theorem:Bdecay} shows that $\lvert f(x)\rvert\le \varepsilon\psi(x)$ if $x\notin K_R$ with $R$ large enough.
	Therefore,
	\begin{equation}
		\mathbb{E}\left[ \lvert f(x(t,x_0)) \rvert\chi_{[\psi(x(t,x_0))>R]} \right]
		\le C\varepsilon\psi(x_0).
	\end{equation}
	Finally, given $R>0$, $\sup_{x\in K_R}\lvert f(x)\rvert<\infty$ by weak continuity.
	By dominated convergence, $\lim_{t\to 0+}\mathbb{E}\left[ \lvert f(x(t,x_0))-f(x_0)\rvert\chi_{[\psi(x(t,x_0))\le R]} \right]=0$.
\qed\end{proof}
\begin{theorem}
	\label{theorem:strongcontprocess}
	Assume \eqref{eq:markovpsibound} for some $C$ and $\varepsilon>0$, and that for any $t>0$, $j\in\mathbb{N}$ and sequence $(x_n)_{n\in\mathbb{N}}$ converging weak-$*$ to some $x_0\in X$, 
	\begin{equation}
		\lim_{n\to\infty}\langle x(t,x_n),y_j\rangle=\langle x(t,x_0),y_j\rangle
		\quad\text{almost surely}. 
	\end{equation}
	Then, $P_t f(x_0):=\mathbb{E}[f(x(t,x_0))]$ satisfies the generalised Feller property and is therefore a strongly continuous semigroup on $\mathcal{B}^\psi(X_{w*})$.
\end{theorem}
\begin{proof}
	Let $f=g\circ h$ with $g\in\mathrm{C}_b^\infty(\mathbb{R}^n)$ and $h(x)=\left( \langle x,y_j\rangle \right)_{j=1,\dots,n}$.
	Such functions are dense in $\mathcal{B}^\psi(X_{w*})$ by Theorem~\ref{theorem:boundedweakcontapproxsep}.
	By Lemma~\ref{lem:Bcharseqwstarcont}, we only have to prove sequential weak-$*$ continuity of $P_tf$ for $ f\in\mathcal{B}^{\psi}(X_{w*})$.
	From the assumption, $\lim_{n\to\infty}h(x(t,x_n))=h(x(t,x_0))$ for any weak-$*$ converging sequence $(x_n)_{n\in\mathbb{N}}$ with limit $x_0$.
	By dominated convergence, $P_t f\in\mathcal{B}^\psi(X_{w*})$. The result now follows from Lemma~\ref{lem:pwcontatzero} and Theorem~\ref{theorem:Ttstrongcont}.
\qed\end{proof}

\begin{example}
	\label{ex:additivedecomp}
	Let $(x(t,x_0))_{t\ge 0}$ admit a decomposition of the form $x(t,x_0)=x_0+X^0_t$ for all $x_0\in X$. Assume furthermore that $\psi(x+y)\le C\psi(x)\psi(y)$ for some $C>0$ and all $x,y\in X$, and that $\mathbb{E}[\psi(X^0_t)]\le C<\infty$ for $t\in[0,\varepsilon]$. Then,
	\begin{equation}
		\mathbb{E}[\psi(x(t,x_0))] \le C^2\psi(x_0),
	\end{equation}
	and it is easy to see that $(x(t,x_0))_{t\ge 0}$ satisfies the conditions of Theorem~\ref{theorem:strongcontprocess}.

	Suppose $x(t,x_0)=x_0+L_t$, where $L_t$ is a c\`adl\`ag L\'evy process with jumps bounded by some constant $c>0$ in $ X$. Then, by Fernique's theorem \cite[Theorem~4.4]{PeszatZabczyk2007}, it follows that $\mathbb{E}[\exp(\beta\lVert L_t\rVert)]<\infty$ for all $\beta>0$.
	Choosing $\psi(x):=\cosh(\beta\lVert x\rVert)$, we see that $\psi(x+y)\le 2\psi(x)\psi(y)$. We obtain that every c\`adl\`ag L\'evy process on a Hilbert space with bounded jumps induces a strongly continuous semigroup on a $\cosh$-weighted space $\mathcal{B}^{\psi}(X_{w*})$.
\end{example}

The continuity assumptions of Theorem~\ref{theorem:strongcontprocess} are typically not easy to verify directly in the weak-$*$ topology.
The following theorem yields a simpler approach by using a compact embedding in a reflexive setting.
\begin{theorem}
	\label{theorem:strongcontprocesscompact}
	Assume \eqref{eq:markovpsibound} for some $C$ and $\varepsilon>0$ on a separable, reflexive Banach space $ Z $. 
	Let $X$ be another separable, reflexive Banach space with $Z\subset X$ compactly embedded. 
	Furthermore, suppose that the Markov process $(x(t,x_0))_{t\ge 0}$ on $Z$ can be extended to $X$, and that for any $f\in\mathrm{C}_b(X)$, the mapping $x_0\mapsto\mathbb{E}[f(x(t,x_0))]$ is continuous with respect to the norm topology of $X$. 
	Then, $P_t f(z):=\mathbb{E}[f(x(t,z))]$ satisfies the generalised Feller property and is therefore a strongly continuous semigroup on $\mathcal{B}^\psi(Z_{w*})$.
\end{theorem}
\begin{remark}
	Note that for concrete examples, we often work the other way round: First, we prove existence of the process on $X$, then we prove the invariance and continuity properties for $ x(t,z)$ on $ Z $ and $ X $. It is actually a result on preservation of regularity, when showing that $x(t,z)\in Z$ almost surely if $z\in Z$.
\end{remark}
\begin{proof}
	Let $\left\{ w_j\colon j\in\mathbb{N} \right\}\subset X^*$ be a countable set which separates the points of $X$. Then, it also separates the points of $Z$. Let $f=g\circ h$ with $g\in\mathrm{C}_b^\infty(\mathbb{R}^n)$ and $h\colon X\to \mathbb{R}^n$, $h(x)=\left( \langle x,w_j\rangle \right)_{j=1,\dots,n}$. By Theorem~\ref{theorem:boundedweakcontapproxsep}, such functions are dense in $\mathcal{B}^\psi(Z_{w*})$.
	Clearly, $f\in\mathrm{C}_b(X)$, and by assumption, $x_0\mapsto u(x_0):=\mathbb{E}[f(x(t,x_0))]$ is continuous with respect to the norm topology.
	As the embedding $\iota\colon Z\to X$ is compact and $K_R$ is bounded for every $R>0$, we see that $u|_{K_R}$ is sequentially weak-$*$ continuous due to the cylindrical structure of $ f $, and it follows that $u|_{Z}\in\mathcal{B}^\psi(Z_{w*})$ by Lemma~\ref{lem:Bcharseqwstarcont}. Lemma~\ref{lem:pwcontatzero} and Theorem~\ref{theorem:Ttstrongcont} prove the claim.
\qed\end{proof}
\begin{example}
	Continuity in norm topologies, as required in Theorem~\ref{theorem:strongcontprocesscompact}, is often satisfied in applications for stochastic partial differential equations, consider for example \cite[Theorem~7.3.5]{DaPratoZabczyk2002} and \cite[Theorem~9.29]{PeszatZabczyk2007}.
	The classical Rellich-Kondrachov type embedding theorems, see \cite[Th\'eor\`eme IX.16]{Brezis1994}, yield compact embeddings for problems on bounded domains.
\end{example}

If $X$ is a separable Hilbert space with scalar product $\langle\cdot,\cdot\rangle$, the separating set can be chosen to be a countable orthonormal basis $(e_j)_{j\in\mathbb{N}}$.
\begin{theorem}
	\label{theorem:separableHilbertcond}
	Assume \eqref{eq:markovpsibound} for some $C$ and $\varepsilon>0$. Let $X$ be a separable Hilbert space with scalar product $\langle\cdot,\cdot\rangle$ and countable orthonormal basis $(e_j)_{j\in\mathbb{N}}$. Denoting by $\pi_M$ the orthogonal projection onto the span of the first $M$ basis vectors, suppose that for $j\in\mathbb{N}$,
	\begin{equation}
		\lim_{M\to\infty}\sup_{x_0\in X}\psi(x_0)^{-1}\mathbb{E}\left[ \lvert \langle x(t,x_0),e_j\rangle - \langle x(t,\pi_M x_0),e_j\rangle \rvert \right]
		= 0.
	\end{equation}
	Then, the semigroup $(P_t)_{t\ge 0}$ defined by $P_t f(x_0):=\mathbb{E}[f(x(t,x_0))]$ satisfies the generalised Feller property and is therefore  strongly continuous on $\mathcal{B}^\psi(X_{w*})$.
\end{theorem}
\begin{proof}
	For $f$ a bounded and smooth cylinder function with $f=f\circ\pi_N$, consider $g_M:=P_t(f\circ\pi_N)\circ\pi_M$.
	We prove that $g_M$ converges to $P_t(f\circ\pi_N)$.
	For any $x_0\in X$, the smoothness of $f$ yields
	\begin{alignat}{2}{}
		\lvert P_t(f\circ\pi_N)(x_0)-g_M(x_0)\rvert
		&\le \mathbb{E}\left[ \lvert f(\pi_N x(t,x_0))-f(\pi_N x(t,\pi_M x_0)) \rvert \right] \notag \\
		&\le C_f\mathbb{E}\left[ \lVert \pi_N(x(t,x_0)-x(t,\pi_M x_0)) \rVert \right] \notag \\
		&\le C_f\sum_{j=1}^{N}\mathbb{E}\left[ \lvert \langle x(t,x_0),e_j\rangle - \langle x(t,\pi_M x_0),e_j\rangle\rvert \right],
	\end{alignat}
	which shows that $P_t \mathcal{B}^{\psi}(X_{w*})\subset\mathcal{B}^{\psi}(X_{w*})$, see Remark~\ref{rem:boundedweakcontapproxsep}.
	By Lemma~\ref{lem:markovpsibound}, $P_t\in L(\mathcal{B}^{\psi}(X_{w*}))$.
	Again, the result follows from Lemma~\ref{lem:pwcontatzero} and Theorem~\ref{theorem:Ttstrongcont}.
\qed\end{proof}
\begin{example}
	The assumptions of Theorem~\ref{theorem:separableHilbertcond} are satisfied for the stochastic Navier-Stokes equation on the two-dimensional torus with additive noise (see \cite{HairerMattingly2008}). The first estimate in \cite[Theorem~A.3]{HairerMattingly2008} proves the condition of Theorem~\ref{theorem:separableHilbertcond}, where the weight function is $\psi(x)=\exp(\eta\lVert x\rVert^2)$ with $\eta>0$ chosen in such a way that $\mathbb{E}[\psi(x(t,x_0))]\le K\psi(x_0)$ for small $t$.
\end{example}

\section{Differentiable functions with controlled growth}

In this section we show an easy way how to construct elements of $ \mathcal{B}^{\psi}(X_{w*}) $ where we actually can hope for short time asymptotics. This is nothing else than including $\mathrm{C}^k$-concepts into the setting of functions $ f $ with growth controlled by $ \psi $.

Let $\mathrm{C}^k(X;Z)$, with $Z$ another Banach space, denote the functions which are $k$-times Fr\'echet differentiable and continuous in the norm topology, together with their derivatives. We introduce spaces $\mathcal{B}^\psi_k(X_{w*})$ of $\mathrm{C}^k$-differentiable functions with derivatives which are in some sense in $\mathcal{B}^\psi(X_{w*})$.
Consider seminorms
\begin{equation}
	\lvert f\rvert_{\psi,j}:=\sup_{x\in X}\psi(x)^{-1}\lVert D^j f(x)\rVert_{L(X^{\otimes j};\mathbb{R})},
\end{equation}
where for a multilinear form $b\colon X^j\to Z$ with $Z$ a Banach space with norm $\lVert\cdot\rVert_Z$,
\begin{equation}
	\lVert b\rVert_{L(X^{\otimes j};Z)}:=\sup_{x_1,\dots,x_j\in X}\lVert x_1\rVert^{-1}\dotsm\lVert x_j\rVert^{-1}\cdot\lVert b(x_1,\dots,x_j)\rVert_Z.
\end{equation}
A fundamental condition simplifying the consideration of such spaces of differentiable functions will be that
\begin{equation}
	\label{eq:CrKR}
	\text{for all $r>0$, there exists $R>0$ such that $C_r(0)\subset K_R$}.
\end{equation}
\begin{definition}
	Let $(X_{w*},\psi)$ be a weighted space satisfying \eqref{eq:CrKR}.

	We say that $f\in \mathcal{B}^\psi_k(X_{w*})$ if and only if $f\in\mathcal{B}^\psi(X_{w*})$, $f\in\mathrm{C}^k(X)$, and for $j=1,\dots,k$, 
	\begin{enumerate}
		\item
			$\lvert f\rvert_{\psi,j}<\infty$,
		\item
			$\lim_{R\to\infty}\sup_{x\in X\setminus K_R}\psi(x)^{-1}\lVert D^j f(x)\rVert_{L(X^{\otimes j};\mathbb{R})}=0$, and
		\item
			for $r>0$, the mapping 
			\begin{equation}
				C_r(0)\times C_1(0)^{j}\to\mathbb{R}, 
				\quad
				(x,x_1,\dots,x_j)\mapsto D^j f(x)(x_1,\dots,x_j)
			\end{equation}
			is continuous with respect to the weak-$*$ topology.
	\end{enumerate}
\end{definition}
\begin{remark}
	The continuity assumption here does not follow from the assumption $f\in\mathrm{C}^k(X)$, as this only guarantees continuity with respect to the norm topology, but we require continuity with respect to the weak-$*$ topology.
	Note that the continuity of $D^j f(x)$ in the last $j$ variables extends to the entire space due to linearity.
\end{remark}
Clearly, $\lVert f\rVert_{\psi,k}:=\lVert f\rVert_{\psi}+\sum_{j=1}^{k}\lvert f\rvert_{\psi,j}$ defines a norm on $\mathcal{B}^\psi_k(X_{w*})$.
Note that $\mathcal{B}^\psi_0(X_{w*})=\mathcal{B}^\psi(X_{w*})$ by Lemma~\ref{lem:Bcharseqwstarcont}.
We easily see that $\mathcal{B}^\psi_{k+1}(X_{w*})$ is continuously embedded in $\mathcal{B}^\psi_{k}(X_{w*})$ for any $k\ge 0$.
\begin{remark}
	\label{rem:smoothcylinderfunctionsdenseBk}
	As the set of cylindrical, $\mathrm{C}^{\infty}$-bounded functions is contained in $\mathcal{B}^\psi_k(X_{w*})$ for any $k\ge 0$ and dense in $\mathcal{B}^\psi(X_{w*})$, we see that $\mathcal{B}^\psi_k(X_{w*})$ is dense in $\mathcal{B}^\psi(X_{w*})$, as well.
\end{remark}
\begin{theorem}
	Consider the weight function $\psi^{(j)}(x,x_1,\dots,x_j):=\psi(x)$ on $X\times C_1(0)^j$.
	Then, $f\in\mathcal{B}^{\psi}_k(X_{w*})$ if and only if $f\in\mathcal{B}^{\psi}(X_{w*})\cap\mathrm{C}^k(X)$ and $D^j f\in\mathcal{B}^{\psi^{(j)}}( (X\times C_1(0)^j)_{w*})$.
\end{theorem}
\begin{proof}
	The first direction is obvious, as $\lvert f\rvert_{\psi,j}=\lVert D^j f\rVert_{\psi^{(j)}}$.
	The other direction also follows from this fact together with Theorem~\ref{theorem:Bdecay} and condition \eqref{eq:CrKR}.
\end{proof}
\begin{theorem}
	With the norm $\lVert\cdot\rVert_{\psi,k}$, $\mathcal{B}^\psi_k(X_{w*})$ is a Banach space.
\end{theorem}
\begin{proof} 
	Let $f_n\in\mathcal{B}^{\psi}_k(X_{w*})$, $n\in\mathbb{N}$, be a Cauchy sequence. 
	Using the last Theorem, we see that $f_n$ converges to some limit $g\in\mathcal{B}^{\psi}(X)$, and similarly $D^j f$ converges to some limit $g_j\in\mathcal{B}^{\psi^{(j)}}( (X\times C_1(0)^j)_{w*} )$, $j=1,\dots,k$.
	As this convergence is uniform on $C_r(0)\times C_r(0)^j$, it follows that $g\in\mathrm{C}^k(X)$ and $D^j g=g_j$.
	In particular, $f_n\to g$ in $\mathcal{B}^{\psi}_k(X_{w*})$, which proves the claim.
\qed\end{proof}
The following result gives conditions for the directional differentiability of a function $f\in\mathcal{B}^{\psi}_k(X_{w*})$ along a vector field defined on a subspace $Z$ of $X$.
\begin{definition}
	Let $X$, $Z$ be dual spaces, $Z\subset X$, and suppose that $(X_{w*},\psi)$, $(Z_{w*},\tilde{\psi})$ are weighted spaces both satisfying \eqref{eq:CrKR}.

	We say that $\sigma\in\mathcal{V}^{\ell}_k( (Z_{w*},\tilde{\psi});(X_{w*},\psi))$ if and only if 
	\begin{enumerate}
		\item 
			$\sigma\in\mathrm{C}^{\ell}(Z;X)$,
		\item
			for $r>0$, the mapping 
			\begin{equation}
				\widetilde{C}_r(0)\times\widetilde{C}_1(0)^{j}\to X_{w*}, 
				\quad
				(x,x_1,\dots,x_j)\mapsto D^j \sigma(x)(x_1,\dots,x_j)
			\end{equation}
			is weak-$*$ continuous, and
		\item
			there exists a function $\varphi\colon Z\to[1,\infty)$ and a constant $C>0$ such that
			$\psi(x)\varphi(x)^{k} \le C\tilde{\psi}(x)$,
			$\lVert \sigma(x)\rVert\le \varphi(x)$ and
			$\lVert D^j\sigma(x)\rVert_{L(Z^{\otimes j};X)} \le \varphi(x)$ for $j=1,\dots,\ell$ and all $x\in Z$.
	\end{enumerate}
\end{definition}
\begin{remark}
	Assuming, for example, that the $\sigma_j$ are sequentially weak-$*$ continuous and bounded together with their derivatives, we see that the choice $Z=X$, $\tilde{\psi}=\psi$ is possible.
\end{remark}
\begin{remark}
	While the definition of $\mathcal{V}^{\ell}_{k}( (Z_{w*},\tilde{\psi}); (X_{w*},\psi) )$ and $\mathcal{B}^{\psi}_k(X_{w*})$ are quite similar, it is not possible to reduce differentiable vector fields with growth control entirely to differentiable functions with growth control.
\end{remark}
\begin{remark}
	\label{rem:vfphisimultaneous}
	Note that if $\sigma_1,\dots,\sigma_k\in\mathcal{V}^{\ell}_{k}( (Z_{w*},\tilde{\psi});(X_{w*},\psi) )$, we can use a single function $\varphi$.
	Indeed, let $\varphi_1,\dots,\varphi_k$ be the respective functions.
	Then, the choice $\varphi(x):=\max_{j=1,\dots,k}\varphi_j(x)$ is admissible for all $\sigma_j$ simultaneously.
\end{remark}
\begin{theorem}
	\label{theorem:Bpsikvfmapping}
	Given $k\ge 1$, $\ell\ge 0$.
	With $X$, $Z$ dual spaces, $Z\subset X$, let $(X_{w*},\psi)$ and $(Z_{w*},\tilde{\psi})$ be weighted spaces.
	Assume that $f\in\mathcal{B}^{\psi}_{k+\ell}(X_{w*})$ and that the vector fields satisfy $\sigma_j\in\mathcal{V}^{\ell}_{k}( (Z_{w*},\tilde{\psi});(X_{w*},\psi) )$.
	Then, 
	\begin{alignat}{2}
		D^k f(\cdot)(\sigma_1(\cdot),\dots,\sigma_k(\cdot)) & \in\mathcal{B}^{\tilde{\psi}}_{\ell}(Z_{w*}),
		\\
		\lVert D^k f(\cdot)(\sigma_1(\cdot),\dots,\sigma_k(\cdot)) \rVert_{\tilde{\psi}}
		&\le
		C^{-1}\lvert f\rvert_{\psi,k}
		\quad\text{and} \\
		\lvert D^k f(\cdot)(\sigma_1(\cdot),\dots,\sigma_k(\cdot)) \rvert_{\tilde{\psi},j}
		&\le
		C_{k,j}\sum_{\iota=0}^{j}\lvert f\rvert_{\psi,k+\iota},
		\quad j=1,\dots,\ell.
	\end{alignat}
	In particular, the linear mapping 
	\begin{equation}
		\mathcal{B}^{\psi}_{k+\ell}(X_{w*})\to\mathcal{B}^{\tilde{\psi}}_{\ell}(Z_{w*}), 
		\quad
		f\mapsto D^k f(\cdot)(\sigma_1(\cdot),\dots,\sigma_k(\cdot))
	\end{equation}
	is continuous.
\end{theorem}
\begin{remark}
	Theorem~\ref{theorem:Bpsikvfmapping} yields another reason why we have to use unbounded weight functions $\psi$.
	Even in the finite-dimensional case, the vector fields defining a stochastic differential equation generally grow linearly.
	Therefore, we need to absorb the growth of the vector fields in the weight $\tilde{\psi}$, and cannot work in an unweighted space such as $\mathrm{C}_b(X)$.
\end{remark}
\begin{proof}
	Define $\tilde{K}_R:=\left\{ z\in Z\colon\tilde{\psi}(z)\le R \right\}$, and choose $\varphi$ as explained in Remark~\ref{rem:vfphisimultaneous}.
	As $K:=\bigcup_{j=1,\dots,d}\sigma_j(\tilde{K}_R)\subset X$ is weak-$*$ compact by the weak-$*$ continuity of $\sigma_j$, $j=1,\dots,d$, it is clear that for $g:=D^k f(\cdot)(\sigma_1(\cdot),\dots,\sigma_k(\cdot))$, $g|_{\tilde{K}_R}$ is weakly continuous, and
	\begin{equation}
		\label{eq:Bpsikvfmapping-pf-funcest}
		\tilde{\psi}(x)^{-1}\lvert g(x)\rvert
		\le
		C^{-1}\psi(x)^{-1}\lVert D^k f(x)\rVert_{L(X^{\otimes k};\mathbb{R})}.
	\end{equation}
	From this, it follows that $\sup_{\tilde{\psi}(x)>R}\tilde{\psi}(x)^{-1}\lvert g(x)\rvert$ tends to zero for $R\to\infty$:

	Assume otherwise.
	Then, there exists $\varepsilon>0$ and a sequence of points $(x_n)_{n\in\mathbb{N}}$ with $\tilde{\psi}(x_n)\ge n$ and $\tilde{\psi}(x_n)^{-1}\lvert g(x_n)\rvert\ge\varepsilon$ for all $n\in\mathbb{N}$.
	We distinguish two cases:
	First, assume that $\limsup_{n\to\infty}\psi(x_n)=\infty$.
	By \eqref{eq:Bpsikvfmapping-pf-funcest}, it follows from 
	\begin{equation}
		\lim_{R\to\infty}\sup_{\psi(x)>R}\psi(x)^{-1}\lVert D^k f(x)\rVert_{L(X^{\otimes k};\mathbb{R})} = 0
	\end{equation}
	that $\liminf_{n\to\infty}\tilde{\psi}(x_n)^{-1}\lvert g(x_n)\rvert=0$, a contradiction.
	Assume now that we have the bound $\psi(x_n)\le K$ for all $n\in\mathbb{N}$ with some $K>0$.
	Then, as $f\in\mathcal{B}^\psi_k(X)$, there exists some constant $C_f$ depending on $f$, but not on $n$ such that
	\begin{equation}
		\tilde{\psi}(x_n)^{-1}\lvert g(x_n)\rvert
		\le
		C_f\tilde{\psi}(x_n)^{-1}\psi(x_n)
		\le
		C_f Kn^{-1},
	\end{equation}
	again a contradiction.
	We obtain $g\in\mathcal{B}^{\tilde{\psi}}(Z)$ by Theorem~\ref{theorem:Bdecay}.

	Consider now $Dg$.
	We have
	\begin{alignat}{2}{}
		&Dg(x)(x_1)
		=
		D^{k+1}f(x)(\sigma_1(x),\dots,\sigma_k(x),x_1) \notag \\
		&\quad+ \sum_{j=1}^{k}D^k f(x)(\sigma_1(x),\dots,\sigma_{j-1}(x),D\sigma_j(x)(x_1),\sigma_{j+1}(x),\dots,\sigma_k(x)).
	\end{alignat}
	This shows that for $r>0$, $Dg|_{\widetilde{C}_r(0)^2}$ is continuous.
	Moreover,
	\begin{alignat}{2}{}
		&\lvert Dg(x)(x_1)\rvert
		\le
		\psi(x)\varphi(x)^k\psi(x)^{-1}\lVert x_1\rVert \times \notag \\
		&\quad\quad
		\times\left( \lVert D^{k+1}f(x)\rVert_{L(X^{\otimes k+1};\mathbb{R})} + k\lVert D^{k}f(x)\rVert_{L(X^{\otimes k};\mathbb{R})} \right) \notag \\
		&\quad\le
		\tilde{\psi}(x)\psi(x)^{-1}\lVert x_1\rVert \left( \lVert D^{k+1}f(x)\rVert_{L(X^{\otimes k+1};\mathbb{R})} + k\lVert D^{k}f(x)\rVert_{L(X^{\otimes k};\mathbb{R})} \right),
	\end{alignat}
	which yields
	\begin{alignat}{2}{}
		\tilde{\psi}(x)^{-1}
		&\lVert Dg(x)\rVert_{L(X;\mathbb{R})} \notag \\
		&\le
		\psi(x)^{-1}\left( \lVert D^{k+1}f(x)\rVert_{L(X^{\otimes k+1};\mathbb{R})} + k\lVert D^{k}f(x)\rVert_{L(X^{\otimes k};\mathbb{R})} \right).
	\end{alignat}
	Similarly as above, we prove $Dg\in\mathcal{B}^{\tilde{\psi}}_1(Z)$.
	Estimates for higher derivatives are obtained in a similar way.
\qed\end{proof}

\section{Numerics Of Stochastic Partial Differential Equations}
\label{sec:applnumericsspde}
We shall show now how the above perspective can be used to obtain rates of convergence for splitting schemes applied to stochastic partial differential equations. Our applications are for general Da Prato-Zabczyk equations \cite{DaPratoZabczyk2002} where the generator admits a compact resolvent and generates a pseudocontractive semigroup. In the next section the Heath-Jarrow-Morton equation of interest rate theory on an adequate Hilbert space \cite{GoldysMusiela2001,FilipovicTeichmann2004} is treated as an example.

Consider a Markov process $x(t,x_0)$ on a Hilbert space $X$.
The basic approach in all our model problems is the following:
\begin{enumerate}
	\item 
		We identify families $(\psi_i)_{i\in I}$ of plausible weight functions and $(Z_j)_{j\in J}$ of suitable subspaces $Z_j\subset X$ of the state space $ X $. This is done in such a way that $(P_t)_{t\ge 0}$ will satisfy $P_t\mathcal{B}^{\psi_i}(Z_j)\subset\mathcal{B}^{\psi_i}(Z_j)$ and $P_t\mathcal{B}^{\psi_i}_k(Z_j)\subset\mathcal{B}^{\psi_i}_k(Z_j)$ for $i\in I$ and $j\in J$.
	\item 
		We split up the generator $\mathcal{G}$ of $P_t$ into a sum of simpler operators $\mathcal{G}_\gamma$, $\gamma=0,\dots,d$ such that each of these operators generates a Markov process on $ X $ and $ Z_j $ with expectation operator $(P^\gamma_t)_{t\ge 0}$, and that these Markov processes can be relatively easily generated.
	\item 
		Using Theorem~\ref{theorem:Bpsikvfmapping}, we can rewrite $\mathcal{G}_\gamma$ on $\mathcal{B}^{\psi_i}_k(Z_j)$ as a sum of directional derivatives along vector fields, which continuously maps $\mathcal{B}^{\psi_i}_k(Z_j)$ to $\mathcal{B}^{\psi_\iota}_{\kappa}(Z_\mu)$.
	\item 
		Together with the results of \cite{HansenOstermann2009}, this proves optimal rates of convergence of the Ninomiya-Victoir splitting scheme or related methods for functions $f\in\mathcal{B}^{\psi_i}_k(Z)$.
\end{enumerate}
Note that for simplicity and ease of representation, we restrict ourselves here to equations driven by Brownian motions.
It is possible to deal with more general L\'evy driving processes in a similar manner, cf.~also \cite{TanakaKohatsuHiga2009}.

Consider a stochastic partial differential equation of Da Prato-Zabczyk type 
\begin{equation}
	\label{eq:dpzspde}
	\dd x(t,x_0) = (A+\alpha(x(t,x_0)))\dd t + \sum_{j=1}^{d}\sigma_j(x(t,x_0))\dd W^j_t
\end{equation}
on a separable Hilbert space $X$ with norm $\lVert\cdot\rVert$, where $\alpha$, $\sigma_j\colon X\to X$ are Lip\-schitz continuous, $(W^j_t)_{j=1,\dots,d}$ is a $d$-dimensional Brownian motion and $A$ with domain $\dom A$ generates a strongly continuous, \emph{pseudocontractive} semigroup on $X$.

Assume furthermore that $A$ has a compact resolvent, and that $\alpha$ and $\sigma_j$ are Lipschitz continuous $\dom A^\ell\to\dom A^\ell$, $\ell=1,\dots,m$, as well, where $\dom A^\ell$ is a Hilbert space with respect to the norm $\lVert x\rVert_{\dom A^\ell}:=\left( \sum_{k=0}^{\ell}\lVert A^k x\rVert^2 \right)^{1/2}$.
Therefore, we can consider the equation to be of Da Prato-Zabczyk type on any of the spaces $\dom A^{\ell}$, $\ell=0,\dots,m$.
\cite[Theorem~7.3.5]{DaPratoZabczyk2002} yields that $\mathbb{E}[(1+\lVert x(t,x_0)\rVert_{\dom A^{\ell}}^2)^{s/2}]\le K(1+\lVert x\rVert_{\dom A^{\ell}}^2)^{s/2}$ for $s\ge 2$, $\ell=0,\dots,m$ and $t\in[0,\varepsilon]$ for some $\varepsilon>0$.
\begin{lemma}
	$\dom A^{\ell+1}$ is compactly embedded in $\dom A^{\ell}$, $\ell\ge 0$.
\end{lemma}
\begin{proof}
	As $A$ has a compact resolvent and generates a strongly continuous semigroup, there exists some $\lambda_0\in\mathbb{R}$ such that $\lambda_0-A$ is continuously invertible and $(\lambda_0-A)^{-1}\colon X\to X$ is compact.
	Clearly, $(\lambda_0-A)^{\ell}\colon\dom A^{\ell}\to X$ is continuously invertible.

	If a sequence $(x_n)_{n\in\mathbb{N}}$ converges weakly in $\dom A^{\ell+1}$ to some $x\in\dom A^{\ell+1}$, then $(\lambda_0-A)^{\ell+1}x_n$ converges weakly to $(\lambda_0-A)^{\ell+1}x$.
	It follows by the compactness of $(\lambda_0-A)^{-1}$ that $(\lambda_0-A)^{\ell}x_n$ converges strongly to $(\lambda_0-A)^{\ell}x$.
	The claim follows.
\qed\end{proof}

Consider the weight functions 
\begin{alignat}{2}{}
	\psi_{\ell,s}\colon&\dom A^{\ell}\to(0,\infty), 
	\\
	&x\mapsto \psi_{\ell,s}(x):=(1+\lVert x\rVert_{\dom A^{\ell}}^2)^{s/2},
	\quad
	\text{$s\ge 2$, $\ell\ge 0$}.
\end{alignat}
Due to reflexivity, the weak and weak-$*$ topology on $\dom A^{\ell}$ agree.
As $t\to x(t,x_0)$ is clearly right continuous and $X$, $\dom A^{\ell}$ are reflexive, Theorem~\ref{theorem:strongcontprocesscompact} proves that the Markov semigroup $(P_t)_{t\ge 0}$ defined through $(x(t,x_0))_{t\ge 0}$ is strongly continuous on $\mathcal{B}^{\psi_{\ell,s}}( (\dom A^{\ell})_{w})$, $\ell=1,\dots,m$.

The following theorem makes substantial use of the assumption that $ A $ generates a pseudocontractive semigroup.
\begin{theorem}
	\label{theorem:dpzeq-pseudocontractivity}
	If $\alpha$ and $\sigma_j$ are Lipschitz continuous on $\dom A^{\ell}$, then 
	\begin{equation}
		\lVert P_t\rVert_{L(\mathcal{B}^{\psi_{\ell,s}}( (\dom A^{\ell})_{w}))}\le\exp(\omega t)
		\quad\text{for some $\omega>0$}.
	\end{equation}
\end{theorem}
\begin{remark}
	The proof is somehow twisted in infinite dimension and does not follow the usual finite dimensional lines of proving that the local martingale part of $ \psi_{\ell,s}(x(t,x_0)) $ is in fact a martingale, and therefore Ito's formula yields the result: 
	we use the Sz\H{o}kefalvi-Nagy theorem \cite[Theorems~7.2.1 and 7.2.3]{Davies1976} and move to a larger Hilbert space $ H \subset \mathcal{H} $ containing $ H $ as a closed subspace and where we can write the solution process $ x(t,x_0) = \pi \mathcal{U}_t Y(t,x_0) $ as orthogonal projection.
\end{remark}

\begin{proof}
	We proceed similarly as in \cite{Teichmann2009}.
	Take $ \ell = 0 $ without any restriction and set $ \psi = \psi_{0,s}$. 
	Additionally we assume that $ A $ generates a contractive semigroup on $H$ by adding the growth to $ \alpha $. 
	Let us consider a larger Hilbert space $ H \subset \mathcal{H} $, where the semigroup generated by $ H $ lifts to a unitary group $ \mathcal{U} $. The projection onto $ H $ is denoted by $ \pi $. Then we consider
	the stochastic partial differential equation prolonged to $ \mathcal{H} $
	\begin{equation}
		\label{eq:dpzspde_proxy}
		\dd X(t,x_0) = (\mathcal{A}X(t,x_0) +\alpha(\pi(X(t,x_0))))\dd t + \sum_{j=1}^{d}\sigma_j(\pi (X(t,x_0)))\dd W^j_t,
	\end{equation}
	where $\mathcal{A}$ is the extension of $A$ to $\mathcal{H}$. 
	By switching to a ``coordinate system'' which moves with velocity $ x \mapsto \mathcal{A} x $, we obtain a new stochastic differential equation
	\begin{equation}
		\dd Y (t,x_0) = \beta(t,Y(t,y_0))\dd t + \sum_{j=1}^{d}\eta_j(t,Y(t,x_0))\dd W^j_t,
	\end{equation}
	with Lipschitz continuous vector fields 
	\begin{alignat}{2}{}
		\beta(t,y)  &= \mathcal{U}_{-t} \alpha(\pi \mathcal{U}_t y) 
		\quad\text{and} \\
		\eta_j(t,y) &= \mathcal{U}_{-t} \sigma_j(\pi \mathcal{U}_t y)
		\quad\text{for $ t \in [0,\varepsilon] $ and $ y \in H $}.
	\end{alignat}

	With \cite[Theorem~7.3.5]{DaPratoZabczyk2002} we can conclude that $ \sup_{t \in [0,\varepsilon]} \mathbb{E}[\lVert Y(t,x_0) \rVert^p]<\infty $ for $p\ge 2$ and $\varepsilon>0$ small.
	Ito's formula applied to 
	\begin{equation} 
		\psi_{\mathcal{H}}(Y(t,x_0)) := {(1+ \lVert Y(t,x_0) \rVert^2)}^{p/2}
	\end{equation}
	together with linear growth and Gronwall's inequality then yields the result; more precisely, defining 
	\begin{equation}
		\mathcal{L}_t f(x):=Df(x)\cdot\beta(t,x) + \frac{1}{2}\sum_{j=1}^{d}D^2 f(x)(\eta_j(t,x),\eta_j(t,x)),
	\end{equation}
	we see that
	\begin{align}
		\mathbb{E}[\psi_{\mathcal{H}}(Y(t,x_0))] = &  \psi_{\mathcal{H}}(x_0) + \int_0^t \mathbb{E}[\mathcal{L}_t(\psi_{\mathcal{H}})(Y(s,x_0)] ds \notag \\
		\leq & \psi(x_0) + \omega \int_0^t \mathbb{E}[\psi_{\mathcal{H}}(Y(s,x_0))] ds,
	\end{align}
	where the constant $ \omega $ depends on the Lipschitz and growth bounds of the vector fields $ \alpha $ and $ \sigma_j $. Noting that $\psi(x_0) = \psi_{\mathcal{H}}(x_0) $, we consider $ x(t,x_0) = \pi \mathcal{U}_t Y(t,x_0) $ and realise -- due to $ \lVert \pi \mathcal{U}_t \rVert \leq 1 $ -- that
	\begin{equation}
		\mathbb{E}[\psi(x(t,x_0)] \leq \mathbb{E}[\psi_{\mathcal{H}}(Y(t,x_0))]] \leq \exp(\omega t) \psi_{\mathcal{H}}(x_0)= 
		\exp(\omega t) \psi(x_0),
	\end{equation}
	which is the desired result.
\qed\end{proof}
\begin{remark}
        Note that under the assumption that the semigroup generated by $A$ consists of compact operators, a condition that is in general stronger than the existence of a compact resolvent (see \cite[Theorem 2.3.2]{Pazy1983}), $(P_t)_{t\ge 0}$ will also be strongly continuous on $\mathcal{B}^{\psi_{0,s}}(X_w)$, $s\ge 2$, by an argument as in \cite[Theorem~2.2]{MaslowskiSeidler1999}.
        $A$ is nevertheless unbounded on $X$, which means that estimates using $\mathcal{B}^{\psi_{\ell,s}}( (\dom A^{\ell})_w )$ are still mandatory.
\end{remark}

Consider now two splitting scheme for \eqref{eq:dpzspde}: the Euler scheme (in an geometric integrator version), and the Ninomiya-Victoir scheme. Assuming that the vector fields $\sigma_j$ are continuously differentiable with bounded first derivative, we switch to Stratonovich form and define $z^0(t,x)$, $z^j(t,x)_t$, $j=1,\dots,d$ as the solutions of
\begin{alignat}{2}
	\frac{\dd}{\dd t} z^0(t,x_0) 
	&= Az^0(t,x_0)+\alpha(z^0(t,x_0))-\frac{1}{2}\sum_{j=1}^{d}D\sigma_j(z^0(t,x_0))\sigma_j(z^0(t,x_0)) \notag \\
	& = Az^0(t,x_0)+\alpha_{0}(z^0(t,x_0)), \\
	\dd z^j(t,x_0)  &= \sigma_j(z^j(t,x_0))\circ\dd W^j_t
\end{alignat}
for all $j=1,\dots,d$, where $\alpha_{0}(z):=\alpha(z)-\frac{1}{2}\sum_{j=1}^{d}D\sigma_j(z)\sigma_j(z)$ is the Stra\-to\-no\-vich-corrected drift.
The use of Stratonovich integrals is not mandatory in our setting as it is in approaches guided by Lyons-Victoir cubature \cite{LyonsVictoir2004,BayerTeichmann2008,NinomiyaVictoir2008}, but it is very helpful~-- the processes $z^j(t,x)$ are, for $j=1,\dots,d$, given through evaluation of the flow of the vector field $\sigma_j$ at random times given by $W^j_t$:
$z^j(t,x)=\mathrm{Fl}^{\sigma_j}_{W^j_t}(x)$, where $\mathrm{Fl}^{\sigma_j}_s$ denotes the flow defined by $\sigma_j$.
Note that only the equation for $Z^{0,x}_t$ contains the unbounded operator $A$, but that this equation is a deterministic evolution equation on $X$.

By Theorem~\ref{theorem:dpzeq-pseudocontractivity}, all Markov semigroups $P^j_t$ are simultaneously strongly continuous on $\mathcal{B}^{\psi_{\ell,s}}( (\dom A^{\ell})_w )$, and $\lVert P^j_t\rVert_{L(\mathcal{B}^{\psi_{\ell,s}}( (\dom A^{\ell})_w ))}\le\exp(\omega_j t)$ with some constants $\omega_j\in\mathbb{R}$, $j=0,\dots,d$.
\begin{remark}
	For the split semigroups, we can also prove pseudocontractivity directly without invoking the Sz\H{o}kefalvi-Nagy theorem.
	Indeed, for $P^j_t$, $j=1,\dots,d$, we can apply It\^o's formula.
	For $P^0_t$, we use the mild formulation 
	\begin{equation}
		z^0(t,x_0)
		=
		\exp(tA)x_0+\int_{0}^{t}\exp( (t-s)A )\alpha(z^0(s,x_0))\dd s,
	\end{equation}
        where $ \exp(At) $ denotes the semigroup generated by $ A $ at time $ t $.
	As $A$ is pseudocontractive, we can assume without loss of generality that $A$ is contractive by modifying $\alpha_0$ by a constant times the identity.
	Together with the Lipschitz continuity of $\alpha$ with constant denoted by $L$, this yields
	\begin{equation}
		\lVert z^0(t,x)\rVert
		\le \lVert x_0\rVert + \int_{0}^{t}\lVert\alpha(z^0(s,x_0))\rVert\dd s
		\le \lVert x_0\rVert + \int_{0}^{t}L\lVert z^0(s,x_0)\rVert\dd s.
	\end{equation}
	The Gronwall inequality proves the required estimate.

	Note that this, together with the fact that the split semigroups approximate $P_t$ strongly on $\mathcal{B}^{\psi}(X_{w*})$ (see Corollary~\ref{cor:strongconvergence}), yields an alternative proof of Theorem~\ref{theorem:dpzeq-pseudocontractivity}.
\end{remark}

We define now two well-known splitting schemes and prove optimal rates of convergence on spaces of sufficiently smooth functions in our general setting.
\begin{definition}[Euler splitting scheme]
	One step of the Euler splitting scheme is defined as
	\begin{equation}
		Q^{\mathrm{Euler}}_{(\Delta t)} := P^0_{\Delta t} P^1_{\Delta t}\dotsm P^d_{\Delta t},
	\end{equation}
	which is a geometric integrator version of the well-known Euler scheme.
\end{definition}
\begin{definition}[Ninomiya-Victoir splitting scheme]
	One step of the Ni\-no\-miya-Victoir splitting is defined as
	\begin{equation}
		Q^{\mathrm{NV}}_{(\Delta t)} := \frac{1}{2}P^0_{\Delta t/2}\left( P^1_{\Delta t}\dotsm P^d_{\Delta t} + P^d_{\Delta t}\dotsm P^1_{\Delta t} \right)P^0_{\Delta t/2},
	\end{equation}
	which should in theory improve the Euler scheme's weak rate of convergence by one order.
\end{definition}

Let $\mathcal{G}_j$ with domain $\dom \mathcal{G}_j$ be the infinitesimal generator of $(P^j_t)_{t\ge 0}$, where $(P^j_t)_{t\ge 0}$ is considered on $\mathcal{B}^{\psi_{\ell_0,s_0}}( (\dom A^{\ell_0})_w )$ with some fixed $\ell_0\in\{0,\dots, m-1\}$, $s_0\ge 2$. 
The function spaces defined below will be fundamental for proving convergence estimates.
\begin{definition}
	Let $p\ge 1$ be given.
	We say that $f\in\mathcal{M}_T^p$ if and only if $f\in\mathcal{B}^{\psi_{\ell_0,s_0}}( (\dom A^{\ell_0})_w )$, $P_t f\in\dom\mathcal{G}^p\cap\bigcap_{j=0}^d\dom\mathcal{G}_j^p$ for $t\in[0,T]$,
	\begin{alignat}{2}
		&C_f:=\sup_{\substack{t\in[0,T] \\ j_1,\dots,j_p=0,\dots,d}}\lVert \mathcal{G}_{j_1}\dotsm\mathcal{G}_{j_p}P_t f\rVert_{\psi_{\ell_0,s_0}}
		<\infty
		\quad\text{and} \\
		&\quad\mathcal{G}^i P_t f = \Biggl( \sum_{j=0}^{d}\mathcal{G}_j \Biggr)^i P_t f,
		\quad i=1,\dots,p.
	\end{alignat}
\end{definition}
\begin{proposition}
	\label{prop:generalisedhansenostermann}
	Let $Q_{\Delta t}$ be a splitting for $P_{\Delta t}$ of classical order $p$.
	For $f\in\mathcal{M}_T^{p+1}$, the splitting converges of optimal order, that is, with a constant $C_f$ independent of $n\in\mathbb{N}$ and $\Delta t>0$, we have that for $n\Delta t\le T$,
	\begin{equation}
		\lVert P_{n\Delta t}f - Q_{(\Delta t)}^n f\rVert_{\psi}
		\le
		C_f \Delta t^p.
	\end{equation}
\end{proposition}
\begin{proof}
	Set $g:=P_t f\in\dom\mathcal{G}\cap\bigcap_{j=0}^d \mathcal{G}_j$.
	The results in \cite[Proof of Theorem~3.4, Section~4.1, Section~4.4]{HansenOstermann2009} prove existence of a family of linear operators $T_t\colon\mathcal{B}^{\psi_{\ell_0,s_0}}( (\dom A^{\ell_0})_w )\to\mathcal{B}^{\psi_{\ell_0,s_0}}( (\dom A^{\ell_0})_w )$ which are uniformly bounded, that is,
	\begin{equation}
		\sup_{t\in[0,\varepsilon]}\lVert T_t\rVert_{L(\mathcal{B}^{\psi}(\dom A^{\ell_0})_w)}\le C_{\varepsilon}<\infty
		\quad\text{for some $\varepsilon>0$},
	\end{equation}
	such that the short term asymptotic expansions of $P_{\Delta t}g$ and $Q_{(\Delta t)}g$ of order $p$ coincide, i.e.~
	\begin{equation}
		P_{\Delta t}g-Q_{(\Delta t)}g
		=
		\Delta t^{p+1} T_{\Delta t}\mathcal{E}_{p+1} g,
	\end{equation}
	where $\mathcal{E}_{p+1}$ is a linear combination of the operators $\mathcal{G}_{j_1}\dotsm\mathcal{G}_{j_{p+1}}$, $j_1,\dots,j_{p+1}=0,\dots,d$, where we apply that by assumption, $\mathcal{G}^{p+1}$ is itself a linear combination of these operators when applied to $g$.
	Thus,
	\begin{equation}
		\lVert P_{\Delta t}g-Q_{(\Delta t)}g\rVert_{\psi}
		\le
		C_f \Delta t^{p+1} \lVert T_{\Delta t}\rVert_{L(\mathcal{B}^{\psi_{\ell,s}}( (\dom A^{\ell})_w))}
		\le
		C_f \Delta t^{p+1}.
	\end{equation}
	It follows that
	\begin{alignat}{2}{}
		\lVert P_{n\Delta t}f-Q_{(\Delta t)}^n f\rVert_{\psi}
		&\le
		C_f \Delta t^{p+1}\sum_{i=1}^{n}\lVert Q_{(\Delta t)}^j\rVert_{L(\mathcal{B}^{\psi_{\ell,s}}( (\dom A^{\ell})_w))} \notag \\
		&\le
		C_f \Delta t^{p}.
	\end{alignat}
\qed\end{proof}

With respect to the Euler scheme define now
For the Euler and Ninomiya-Victoir schemes, we define $\mathcal{M}^{\mathrm{Euler}}_T\subset\mathcal{B}^{\psi_{\ell,s}}( (\dom A^{\ell})_w )$ 
\begin{equation}
	\mathcal{M}^{\mathrm{Euler}}_T
	:=
	\mathcal{M}_T^{2}
\end{equation}
and $\mathcal{M}^{\mathrm{NV}}_T\subset\mathcal{B}^{\psi_{\ell,s}}( (\dom A^{\ell})_w )$ by
\begin{equation}
	\mathcal{M}^{\mathrm{NV}}_T
	:=
	\mathcal{M}_T^{3}.
\end{equation}
The following results are now an easy consequence of Proposition~\ref{prop:generalisedhansenostermann}.
\begin{corollary}
	For $ f \in \mathcal{M}^{\mathrm{Euler}}_T $ there exists some constant $C_f$ independent of $n\in\mathbb{N}$ and $\Delta t>0$ such that if $n\Delta t\le T$,
	\begin{equation}
		\lVert P_{n\Delta t} f - Q_{(\Delta t)}^n f\rVert_{\psi}
		\le
		C_f \Delta t.
	\end{equation}
	Hence, for $f\in\mathcal{M}^{\mathrm{Euler}}_T$, the Euler splitting scheme converges of optimal order.
\end{corollary}
\begin{corollary}
	For $ f \in \mathcal{M}^{\mathrm{NV}}_T $ there exists some constant $C_f$ independent of $n\in\mathbb{N}$ and $\Delta t>0$ such that if $n\Delta t\le T$,
	\begin{equation}
		\lVert P_{n\Delta t} f - (Q^{\mathrm{NV}}_{(\Delta t)})^n f\rVert_{\psi}
		\le
		C_f \Delta t^2.
	\end{equation}
	Hence, for $f\in\mathcal{M}^{\mathrm{NV}}_T$, the Ninomiya-Victoir splitting scheme converges of optimal order.
\end{corollary}
\begin{remark}
	Note that in principle, we can now also consider different splittings than the Euler or the Ninomiya-Victoir schemes.
	It is, however, not possible to obtain higher rates of convergence due to inherent limits of splitting schemes with positive coefficients (see \cite{BlanesCasas2005}), and positivity of coefficients is mandatory in the probabilistic setting under concern.
\end{remark}

We derive easy conditions guaranteeing $f\in\mathcal{M}^{\mathrm{NV}}_T$.
\begin{lemma}
	\label{lem:dpzeq-vola}
	Suppose that $f\in\mathrm{C}^2(\dom A^{\ell})$, $0\le\ell\le\ell_0$, with uniformly continuous derivatives on bounded sets in $\dom A^{\ell}$.
	Further, assume that $f$, $g\in\mathcal{B}^{\psi_{\ell,s}}( (\dom A^{\ell})_w )$, where
	\begin{equation}
	g:=\frac{1}{2}Df(\cdot)D\sigma_j(\cdot)\sigma_j(\cdot)+D^2 f(\cdot)(\sigma_j(\cdot),\sigma_j(\cdot)).
	\end{equation}
	Then $f\in\dom \mathcal{G}_j$ and $\mathcal{G}_j f=g$.
\end{lemma}
\begin{proof}
	Under the given assumption, we apply It\^o's formula \cite[Theorem~7.2.1]{DaPratoZabczyk2002} to obtain
	\begin{alignat}{2}
		P^j_t f(x) 
		= f(x) + \int_0^t \mathbb{E}\left[ g(z^j(s,x)) \right]\dd s
		=f(x) + \int_{0}^{t}P^j_s g(x)\dd s.
	\end{alignat}
	The result follows from $g\in\mathcal{B}^{\psi_{\ell,s}}( (\dom A^{\ell})_w )$ and the strong continuity of $(P^j_t)_{t\ge 0}$.
\qed\end{proof}
\begin{lemma}
	\label{lem:dpzeq-drift}
	If $f\in\mathrm{C}^1(\dom A^{\ell})$, $0\le\ell\le\ell_0-1$, with $f$, $g:=Df(\cdot)(A\cdot+\alpha_0(\cdot))\in\mathcal{B}^{\psi_{\ell,s}}( (\dom A^{\ell})_w )$, then $f\in\dom \mathcal{G}_0$ and $\mathcal{G}_0 f=g$.
\end{lemma}
\begin{remark}
	For $f\in\mathrm{C}^1(\dom A^{\ell})$, $Df(x)$ defines a continuous functional on $\dom A^{\ell+1}$.
	It follows that $g\colon\dom A^{\ell_0}\to\mathbb{R}$ is well-defined for $\ell\le\ell_0-1$.
\end{remark}
\begin{proof}
	By the fundamental theorem of calculus,
	\begin{equation}
		P^0_t f(x_0) 
		= f(x_0) + \int_{0}^{t} g(z^0(s,x_0))\dd s
		= f(x_0) + \int_{0}^{t}P^0_s g(x_0)\dd s.
	\end{equation}
	Again, $g\in\mathcal{B}^{\psi_{\ell,s}}( (\dom A^{\ell})_w )$ and strong continuity of $(P^j_t)_{t\ge 0}$ prove the result.
\qed\end{proof}
\begin{lemma}
	\label{lem:dpzeq-full}
	Assume that $f\in\mathrm{C}^2(\dom A^{\ell})$, $0\le\ell\le\ell_0-1$, with uniformly continuous derivatives on bounded sets in $\dom A^{\ell}$, and that $f$, $g:=Df(\cdot)(A\cdot+\alpha_0(\cdot))+\sum_{j=1}^{d}\frac{1}{2}Df(\cdot)D\sigma_j(\cdot)\sigma_j(\cdot)+D^2 f(\cdot)(\sigma_j(\cdot),\sigma_j(\cdot))\in\mathcal{B}^{\psi_{\ell,s}}( (\dom A^{\ell})_w )$.
	Then $f\in\dom \mathcal{G}$ and $\mathcal{G} f=g$.
\end{lemma}
\begin{proof}
	It\^o's formula \cite[Theorem~7.2.1]{DaPratoZabczyk2002}, $g\in\mathcal{B}^{\psi_{\ell,s}}( (\dom A^{\ell})_w )$ and the strong continuity of $(P_t)_{t\ge 0}$ yield the result.
\qed\end{proof}

The following result shows how compactness can be used to prove weak continuity of nonlinear mappings.
\begin{proposition}
	\label{prop:dpzeq-weakcontcompactembedding}
	Suppose that $X$, $Z$ are Banach spaces with norms $\lVert\cdot\rVert_{X}$, $\lVert\cdot\rVert_{Z}$, $Z\subset X$ compactly embedded, and $j\ge 1$.
	Let $F\in\mathrm{C}(X^j;X)$ and assume that for some $r>0$, $F(C_r(0)^j)\subset Z$ and is bounded in $Z$, where we set $C_r(0):=\left\{ z\in Z\colon\lVert z\rVert_{Z}\le r \right\}$.

	Then, $F\colon C_r(0)^j\to Z$ is sequentially weakly continuous, i.e., whenever a sequence $(\zeta_n)_{n\in\mathbb{N}}\subset C_r(0)^j$ converges weakly to $\zeta$, it follows that $F(\zeta_n)$ converges weakly to $F(\zeta)$ in $Z$.
\end{proposition}
\begin{proof}
	Denote the compact embedding $\iota\colon Z\to X$, and let $(\zeta_n)_{n\in\mathbb{N}}$, $z_n=(z_{n,1},\dots,z_{n,j})$, converge weakly to $\zeta=(z_1,\cdots,z_j)$ in $Z$.
	By assumption, $\lVert F(\zeta_n)\rVert_{Z}\le C$ for some $C>0$.
	Additionally, $(\zeta_n)_{n\in\mathbb{N}}$ converges to $\zeta$ in the norm of $X^j$.
	The continuity of $F$ on $X$ yields $X$-norm convergence of $(F(\zeta_n))_{n\in\mathbb{N}}$ to $F(\zeta)$.

	As $\iota$ is injective, it follows that its adjoint $\iota^*\colon X^*\to Z^*$ has dense range by \cite[Corollaire II.17(iii)]{Brezis1994}, where $\iota^*$ is given by $(\iota^* x^*)(z)=x^*(\iota z)$ for all $z\in Z$ and $x^*\in X^*$.
	It follows that for every $z^*\in Z^*$ and $\varepsilon>0$, there exists some $x^*\in X^*$ such that $\lVert z^*-\iota x^*\rVert_{L(Z;\mathbb{R})}<\varepsilon$, whence $\lvert z^*(z) - x^*(\iota z)\rvert < \varepsilon$.
	The result follows from the norm convergence of $(F(\zeta_n))_{n\in\mathbb{N}}$ in $X$ and
	\begin{alignat}{2}
		\lvert z^* & (F(\zeta_n) - F(\zeta)) \rvert 
		\le 2C\varepsilon 
		+ \lVert x^*\rVert_{L(X;\mathbb{R})}\cdot\lVert F(\zeta_n) - F(\zeta)\rVert_{X}.
	\end{alignat}
\qed
\end{proof}
\begin{lemma}
	\label{lem:dpzeq-condapplyvf}
	Assume that $\alpha$, $\sigma_j\in\mathrm{C}^{k-2}(X;X)$ with bounded derivatives, and that $0\le\ell\le\ell_0-1$, $2\le s\le s_0-2$ and $k\ge 2$.
	Then, $\mathcal{G}$, $\mathcal{G}_j\colon\mathcal{B}^{\psi_{\ell,s}}_{k}( (\dom A^{\ell})_w )\to\mathcal{B}^{\psi_{\ell+1,s+2}}_{k-2}( (\dom A^{\ell+1})_w )$ are continuous for $j=0,\dots,d$, and
	\begin{equation}
		\sum_{j=0}^{d}\mathcal{G}_j f=\mathcal{G} f
		\quad\text{for all $f\in\mathcal{B}^{\psi_{\ell,s}}_k( (\dom A^{\ell})_w )$}.
	\end{equation}
\end{lemma}
\begin{proof}
	Note that $f$ and its derivatives are uniformly continuous on bounded subsets of $\dom A^{\ell}$ as they are weakly compact.
	By the Lemmas~\ref{lem:dpzeq-vola}, \ref{lem:dpzeq-drift} and \ref{lem:dpzeq-full}, it follows that $\mathcal{B}^{\psi_{\ell,s}}_k( (\dom A^{\ell})_w )\subset\dom \mathcal{G}\cap\bigcap_{j=0}^{d}\mathcal{G}_j$, and that for $f\in\mathcal{B}^{\psi_{\ell,s}}( (\dom A^{\ell})_w )$, $\mathcal{G}$ and $\mathcal{G}_j f$ are given by a sum of directional derivatives
	$\alpha$, $\sigma_j$ and their derivatives are norm continuous on $\dom A^{\ell}$ by assumption.
	By the compact embedding $\dom A^{\ell+1}\to\dom A^{\ell}$, it follows by Proposition~\ref{prop:dpzeq-weakcontcompactembedding} that $\alpha$ and $\sigma_j$ are weakly continuous on every bounded set in $\dom A^{\ell+1}$.
	By linear boundedness with bounded derivatives, we can choose $\varphi(x):=\left( 1+\lVert x\rVert_{\dom A^{\ell+1}}^2 \right)^{1/2}$ to obtain 
	\begin{alignat}{2}{}
		\alpha&\in\mathcal{V}^{k-2}_1\Bigl( ( (\dom A^{\ell+1})_w,\psi_{\ell+1,s+2});( (\dom A^{\ell})_w,\psi_{\ell,s}) \Bigr), 
		\quad\text{and} \\
		\sigma_j&\in\mathcal{V}^{k-2}_2\Bigl( ( (\dom A^{\ell+1})_w,\psi_{\ell+1,s+2});( (\dom A^{\ell})_w,\psi_{\ell,s}) \Bigr), 
		\quad
		\text{$j=1,\dots,d$}.
	\end{alignat}
	Using Theorem~\ref{theorem:Bpsikvfmapping}, we see that $\mathcal{G}_j f\in\mathcal{B}^{\psi_{\ell+1,s+2}}_{k-2}( (\dom A^{\ell+1})_w )$.
\qed\end{proof}

\begin{lemma}
	\label{lem:dpzeq-derbound}
	Suppose $\alpha$, $\sigma_j\in\mathrm{C}^k(X;X)$ with bounded derivatives.
	Let $1\le\ell\le\ell_0$, $2\le s\le s_0$ and $k\ge 0$.
	Then, $P_t \mathcal{B}^{\psi_{\ell,s}}_k( (\dom A^{\ell})_w )\subset\mathcal{B}^{\psi_{\ell,s}}_k( (\dom A^{\ell})_w )$, and $\sup_{t\in[0,T]}\lVert P_t f\rVert_{\psi_{\ell,s},k}\le K_T\lVert f\rVert_{\psi_{\ell,s},k}$ with some constant $K_T$ independent of $f$. 
\end{lemma}
\begin{proof}
	The results in \cite[Theorem~5.4.1]{DaPratoZabczyk1996} and \cite[Theorem~7.3.6]{DaPratoZabczyk2002} prove existence of $C>0$ such that $\lVert D_x^j X^x_t\rVert_{L( (\dom A^{\ell})^{\otimes j};\dom A^{\ell})}\le C$ almost surely for all $x\in\dom A^{\ell}$ and $j=1,\cdots,k$, and that these mappings are almost surely norm continuous in $x$.
	By the compact embedding, almost sure sequential weak continuity on bounded sets of $\dom A^{\ell}$ follows from Proposition~\ref{prop:dpzeq-weakcontcompactembedding}.
	We obtain 
	\begin{alignat}{2}
		\lvert DP_t f(x_0)(x_1)\rvert
		&\le
		\lVert x_1\rVert\mathbb{E}[\lVert Dx(t,x_0)\rVert_{L(\dom A^{\ell};\dom A^{\ell})}\times \\
		&\phantom{\le}\times\lVert Df(x(t,x_0))\rVert_{L(\dom A^{\ell};\mathbb{R})}] \notag \\
		&\le
		C_t\lvert f\rvert_{\psi_{\ell,s},1}\psi_{\ell,s}(x)\lVert x_1\rVert
	\end{alignat}
	with some constant $C_t$ independent of $x$ and $f$, and similarly for higher derivatives.
\qed\end{proof}

\begin{theorem}
	Assume that $\alpha$, $\sigma_j\in\mathrm{C}^6(X;X)$ with bounded derivatives, that $\ell_0\ge 4$ and that $s_0\ge 8$.
	Then, for $0\le\ell\le\ell_0-4$ and $2\le s\le s_0-6$, $\mathcal{B}^{\psi_{\ell,s}}_6( (\dom A^{\ell})_w )\subset\mathcal{M}^{\mathrm{NV}}_T$.
	In particular, $\mathrm{C}_b^6(X)\subset\mathcal{M}^{\mathrm{NV}}_T$.
\end{theorem}
\begin{proof}
	By Lemma~\ref{lem:dpzeq-derbound}, $\lVert P_t f\rVert_{\psi,6}\le K_T\lVert f\rVert_{\psi,6}<\infty$ for all $t\in[0,T]$.
	The first claim follows by iterating Lemma~\ref{lem:dpzeq-condapplyvf}.

	For the second claim, let $f\in\mathrm{C}_b^6(X)$.
	$f\in\mathrm{C}^6(\dom A^{\ell})$ is obvious.
	By the compact embedding $\dom A^{\ell}\to X$, $f$ has weakly continuous derivatives on bounded sets of $\dom A^{\ell}$, and Lemma~\ref{lem:Bcharseqwstarcont} proves $f\in\mathcal{B}^{\psi}( (\dom A^{\ell})_{w} )$. 
	Boundedness of the derivatives shows $\lvert f\rvert_{\psi,j}<\infty$ and 
	\begin{equation}
		\lim_{R\to\infty}\sup_{\psi_{\ell,s}(x)>R}\psi(x)^{-1}\lVert D^j f(x)\rVert_{L( (\dom A^{\ell})^j;\mathbb{R} )}=0.
	\end{equation}
	Hence, $f\in\mathcal{B}^{\psi_{\ell,s}}_6( (\dom A^{\ell})_w )$.
\qed\end{proof}
The following theorem follows analogously.
\begin{theorem}
	Assume that $\alpha$, $\sigma_j\in\mathrm{C}^4(X;X)$ with bounded derivatives, that $\ell_0\ge 3$ and that $s_0\ge 6$.
	Then, for $0\le\ell\le\ell_0-3$ and $2\le s\le s_0-4$, $\mathcal{B}^{\psi_{\ell,s}}_4( (\dom A^{\ell})_w )\subset\mathcal{M}^{\mathrm{Euler}}_T$.
	In particular, $\mathrm{C}_b^4(X)\subset\mathcal{M}^{\mathrm{Euler}}_T$.
\end{theorem}
\begin{corollary}
	\label{cor:strongconvergence}
	Let $f\in\mathcal{B}^{\psi}(X_{w*})$.
	Then, for any $t>0$,
	\begin{equation}
		\lim_{n\to\infty}\lVert P_t f-(Q_{(t/n)}^{\mathrm{Euler}})^n f\rVert_{\psi}
		=\lim_{n\to\infty}\lVert P_t f-(Q_{(t/n)}^{\mathrm{NV}})^n f\rVert_{\psi}
		=0,
	\end{equation}
	that is, the Euler and Ninomiya-Victoir splittings converge strongly on the space $\mathcal{B}^{\psi}(X_{w*})$.
\end{corollary}
\begin{proof}
	This follows from the density of bounded, smooth, cylindrical functions in $\mathcal{B}^{\psi}(X_{w*})$, see Remark~\ref{rem:smoothcylinderfunctionsdenseBk}.
\qed\end{proof}
\begin{example}
	Assume that $\alpha\equiv 0$ and that the $\sigma_j$ are constant, $j=1,\dots,d$.
	This includes, in particular, stochastic heat and wave equations on bounded domains with additive noise.
	It is easy to see that if $A\colon\dom A\to X$ admits a compact resolvent, we are in the situation described above, and the Ninomiya-Victoir splitting converges of optimal order.
\end{example}
\begin{example}
	Note that finite-dimensional problems with Lipschitz-con\-ti\-nu\-ous coefficients are also included in this setting.
	Here, $A$ can be chosen to be zero, and the embedding is trivially compact due to the local compactness of finite-dimensional spaces.
\end{example}

\section{An Example: The Heath-Jarrow-Morton Equation Of Interest Rate Theory}
With $\alpha\in\mathbb{R}$ and $w_\alpha:=\exp(\alpha x)$, we set $\mathrm{L}_{\alpha}^2(\mathbb{R}_+):=\mathrm{L}^2(\mathbb{R}_+,w_\alpha)$ and $\mathrm{H}_{\alpha}^k(\mathbb{R}_+):=\mathrm{H}^k(\mathbb{R}_+,w_\alpha)$.
Here and in the following, $\mathbb{R}_+:=(0,\infty)$.
\begin{proposition}
	\label{prop:hjm-H1alphacompact}
	For every $\alpha>0$, the space $\mathrm{H}^1(\mathbb{R}_+)\cap\mathrm{L}_{\alpha}^2(\mathbb{R}_+)$ with norm given by 
	\begin{equation}
		\lVert f\rVert
		:=
		\left(\lVert f\rVert_{\mathrm{H}^1(\mathbb{R}_+)}^2+\lVert f\rVert_{\mathrm{L}_{\alpha}^2(\mathbb{R}_+)}^2\right)^{1/2}
	\end{equation}
	is compactly embedded in $\mathrm{L}^2(\mathbb{R}_+)$.
\end{proposition}
Note that the proof shows that an analogous result holds true for any weight function $w$ with $\lim_{x\to+\infty}w(x)=+\infty$.
\begin{proof}
	We apply \cite[Th\'eor\`eme IV.26]{Brezis1994}.
	For any $\tau>0$,
	\begin{alignat}{2}
		\int_{\mathbb{R}_+}\lvert f(x+\tau)-f(x)\rvert^2\dd x
		&\le
		\int_{\mathbb{R}_+}\int_{0}^{\tau}\lvert f'(x+s)\rvert^2\dd s\dd x \notag \\
		&=
		\int_{0}^{\tau}\int_{\mathbb{R}_+}\lvert f'(x+s)\rvert^2\dd x \dd s \notag \\
		&\le
		\tau\lVert f\rVert_{\mathrm{H}^1(\mathbb{R}_+)},
	\end{alignat}
	and for any $R>0$,
	\begin{alignat}{2}{}
		\int_{R}^{\infty}\lvert f(x)\rvert^2\dd x
		&\le
		\exp(-\alpha R)\int_{R}^{\infty}\lvert f(x)\rvert^2\exp(\alpha x)\dd x \notag \\
		&\le
		\exp(-\alpha R)\lVert f\rVert_{\mathrm{L}_{\alpha}^2(\mathbb{R}_+)}.
	\end{alignat}
	These estimates prove the claim.
\qed\end{proof}
\begin{corollary}
	\label{cor:hjm-HbetaHalphacompact}
	For any $\alpha$, $\beta\in\mathbb{R}$ with $\beta>\alpha$ and integer $k\ge 0$, $\mathrm{H}_{\beta}^{k+1}(\mathbb{R}_+)$ is compactly embedded in $\mathrm{H}_{\alpha}^{k}(\mathbb{R}_+)$.
\end{corollary}
\begin{proof}
	Assume first $k=0$.
	Then, Proposition~\ref{prop:hjm-H1alphacompact} shows that $\mathrm{H}_{\beta-\alpha}^1(\mathbb{R}_+)$ is compactly embedded in $\mathrm{L}^2(\mathbb{R}_+)$.

	The mapping $T\colon\mathrm{L}^2(\mathbb{R}_+)\to\mathrm{L}_{\alpha}^2(\mathbb{R}_+)$, $f\mapsto \exp(-\frac{\alpha}{2} x)f$, is an isometric isomorphism, and $T(\mathrm{H}_{\beta-\alpha}^1(\mathbb{R}_+))=\mathrm{H}_{\beta}^1(\mathbb{R}_+)$, where the norms $\lVert T^{-1}f\rVert_{\mathrm{H}_{\beta-\alpha}^1(\mathbb{R}_+)}$ and $\lVert f\rVert_{\mathrm{H}_{\beta}^1(\mathbb{R}_+)}$ are equivalent.
	It follows that $\mathrm{H}_{\beta}^1(\mathbb{R}_+)$ is compactly embedded in $\mathrm{L}_{\alpha}^2(\mathbb{R}_+)$.
	The full result follows by a simple induction.
\qed\end{proof}
This compact embedding lets us derive rates of convergence of the Ninomiya-Victoir splitting scheme in the HJM setting of \cite{Filipovic2001,FilipovicTeichmann2004} (see also \cite{GoldysMusiela2001} for another setting where our approach should be equally applicable).
There, the space $H_w$ consisting of functions $f$ with $f'$ lying in some weighted Sobolev space is used.
We shall restrict ourselves to exponential weights.
We set $H_\alpha=\left\{ h\in\mathrm{L}^1_{\mathrm{loc}}(\mathbb{R}_+)\colon h'\in\mathrm{L}_\alpha^2(\mathbb{R}_+) \right\}$ for $\alpha>0$ with norm 
\begin{equation}
	\lVert h\rVert_{H_{\alpha}}:=\left( \lvert h(0)\rvert^2 + \int_{\mathbb{R}_+}\lvert h(x)\rvert^2\exp(\alpha x)\dd x \right)^{1/2}.
\end{equation}
Furthermore, we define $H_{\alpha}^0:=\left\{ h\in H_\alpha\colon h(+\infty)=0 \right\}$ (see \cite[Chapter~5]{Filipovic2001}).

Let $\sigma_j\colon H_\alpha\to H_\alpha^0$ be Lipschitz continuous and bounded, $j=1,\dots,d$.
Define the Heath-Jarrow-Morton drift 
\begin{equation}
\alpha_{\mathrm{HJM}}\colon H_\alpha\to H_\alpha,
\quad
\alpha_{\mathrm{HJM}}(h):=\sum_{j=1}^{d}\mathcal{S}\sigma^j(h), 
\end{equation}
where $\mathcal{S}f(x):=f(x)\int_{0}^{x}f(y)\dd y$.
The operator $A:=\frac{\dd}{\dd x}$ with domain $\dom A:=\left\{ h\in H_\alpha\colon h'\in H_{\alpha} \right\}$ is the infinitesimal generator of the shift semigroup on $H_{\alpha}$.
Then, the HJM equation
\begin{alignat}{2}{}
	\label{eq:hjm-eq}
	\dd r(t,r_0)
	&=
	(Ar(tr_0)+\alpha_{\mathrm{HJM}}(r(t,r_0)))\dd t + \sum_{j=1}^{d}\sigma_j(r(t,r_0))\dd W^j_t,
	\\
	r(0,r_0)
	&=
	r_0,
	\notag
\end{alignat}
where $(W^j_t)_{j=1,\dots,d}$ is a $d$-dimensional Brownian motion, has a unique solution (see \cite[Chapter~5]{Filipovic2001}).

Let $A_\beta$ be the restriction of $A$ to $\dom A_\beta:=\left\{ h\in H_\beta\colon h'\in H_\beta \right\}$.
It is clear that $A_\beta$ is the infinitesimal generator of the shift semigroup on $H_\beta$.
We shall assume now in addition that $\alpha_{\mathrm{HJM}}$ and $\sigma_j$, $j=1,\dots,d$ are Lipschitz continuous on $H_\beta$ and $\dom A_\beta^\ell$ for $\ell=1,\dots,m$ with some $m\ge 1$.
Such an assumption is actually not untypical and is even weaker than \cite[(A1), p.~135]{FilipovicTeichmann2004}.
\begin{theorem}
	For any $k\ge 0$, $\dom A_\beta^k$ is compactly embedded in $H_\alpha$.
\end{theorem}
\begin{proof}
	As $\dom A_\beta^k$ is continuously embedded in $\dom A_\beta$ for any $k\ge 1$, we only have to prove the result for $k=1$.
	Let therefore $h_n\in\dom A_\beta$ be a sequence converging weakly to $h\in A_\beta$, that is, $h_n$ and $h_n'$ converge weakly to $h$ and $h'$ in the topology of $H_\beta$.
	Then, as point evaluations are continuous in $H_\beta$, we obtain that $\lim h_n(0)=h(0)$.
	As $h_n'$ converges weakly to $h'$ in $H_\beta$, we see that $h_n'$ and $h_n''$ converge weakly to $h'$ and $h''$ in $\mathrm{L}_\beta^2(\mathbb{R}_+)$, that is, $h_n'$ converges weakly to $h'$ in $\mathrm{H}_\beta^1(\mathbb{R}_+)$.
	By Corollary~\ref{cor:hjm-HbetaHalphacompact}, $h_n'$ converges strongly to $h'$ in $\mathrm{L}_\alpha^2(\mathbb{R}_+)$, and the result follows.
\qed\end{proof}
From Lipschitz continuity of the coefficients, we obtain easily that solutions of \eqref{eq:hjm-eq} depend Lipschitz continuously on the initial value.
Thus, weakly continuous dependence in $\dom A_\beta^k$ follows for any $k\ge 1$.
Similarly as in Section~\ref{sec:applnumericsspde}, by splitting into a part corresponding to the Stratonovich-corrected drift, $\frac{\dd}{\dd x}+\alpha_0$, and the parts corresponding to the diffusions, we obtain optimal weak rates of convergence in a supremum norm weighted by $\psi(h):=(1+\lVert h\rVert_{\dom A_\beta^{\ell}}^2)^{s/2}$, $\ell$, $s$ large enough, for sufficiently smooth functions if $\alpha_{\mathrm{HJM}}$ and $\sigma_j$ are smooth enough.

\section*{Acknowledgements}
The first author thanks Michael Kaltenb\"ack and Georg Grafendorfer for fruitful discussions on early drafts.

\bibliographystyle{spmpsci}
\bibliography{lit}

\end{document}